\index{}
\documentclass[10pt]{article}
\usepackage{amssymb , amsmath, amsthm}
\usepackage{amsmath}
\usepackage{amsbsy}
\usepackage{amssymb}
\usepackage{color}
\usepackage{graphicx}
\usepackage{subfigure}
\usepackage[active]{srcltx}

\newtheorem{theo}{Theorem}[section]

\newtheorem{coro}{Corollary}[section]

\newtheorem{definition}{Definition}[section]

\newtheorem{remark}{Remark}[section]

\def\q{\mathbb Q}
\def\co{\mathbb C}
\def\r{\mathbb R}

\def\h{\mathbb H}

\def\n{\mathbb N}

\newcommand{\R}{\mathbb R}

\newcommand{\wt}{\widetilde}
\newcommand{\wh}{\widehat}

\let\h=\hip

\def\Om{\Omega}

\def \Om{\Omega}

\def\rmd{\mathop{\rm d\kern -1pt}\nolimits}
\def\rme{\mathop{\rm e\kern -1pt}\nolimits}

\def \noi {\noindent}

\def\bel{ \medskip
 \centerline{$ \ast \hbox to 1.0cm{}\ast \hbox to 1.0cm{}\ast $}
}

\def\longerrightarrow{-\kern-5pt\longrightarrow}

\def\star{\lower 1pt\hbox{*}}
\def \nulset {
\raise 1pt\hbox{ \hskip -3pt$\not$\kern -0.2pt \raise
.7pt\hbox{${\scriptstyle\bigcirc}$}}}

\newcommand{\hi}[1]{\mathbb{H}^#1}

\newcommand{\ov}[1]{\overline{#1}}

\let\leq=\leqslant
\let\geq=\geqslant

\title{{\bf Minimal Graphs in  $Nil_3:$ existence and
non-existence results}}

\author{B. Nelli, R. Sa Earp, E. Toubiana}

\date{}

\begin{document}

\maketitle

\begin{abstract}
\noi We study the minimal surface equation in the Heisenberg space, $Nil_3.$ 
A geometric proof of  non existence  of minimal graphs over non convex, bounded and 
unbounded domains is achieved  (our proof  holds in the 
Euclidean space as well). We solve the Dirichlet problem for the minimal surface
equation over bounded and unbounded convex domains, taking   bounded,  piecewise 
 continuous boundary value. We  are able to construct  a Scherk type minimal surface and we use it  
as a barrier to construct  non trivial minimal graphs over a wedge of angle 
$\theta\in [\frac{\pi}{2}, \pi[$ taking
non negative continuous boundary data, having at least quadratic growth. In the case of an half-plane, we are also able  to 
give solutions (with either linear or quadratic growth), provided some geometric hypothesis on the boundary data. Finally,  some  open problem arising from our work, are posed.

\end{abstract}



%\subclass

\section{Introduction}
\label{intro}

In this paper we study the minimal  surface equation in the Heisenberg space, $Nil_3$. 
We first consider non convex (smooth) domains  $\Om\subset \R^2$ and we 
provide a construction of a  continuous (smooth) boundary  data that does not admit a minimal 
extension over $\Om$
 (Theorem \ref{non-existence}, \ref{non-existence-unbounded}). 
 We point out that our proof is geometric in nature and holds for the minimal 
surface equation in the Euclidean three space as well. Notice that 
for bounded non convex domains the non solvability of the Euclidean 
minimal equation in two variables  was established by Finn \cite{F} for 
continuous  boundary  data,  and  by  Jenkins and Serrin
\cite{JS1} for  $C^2$   boundary data, with arbitrarily small absolute value.

\vskip1.5mm

Secondly, we prove that, for any convex domain $\Om$ (bounded or unbounded) different from 
the half-plane,  given a piecewise continuous (smooth) boundary  
data
$\varphi$ over $\partial \Om$, there exists a minimal extension  $u$ 
of $\varphi$ over $\Om.$ Moreover, we prove that the boundary of the graph of 
$u$ is the union of the graph of $\varphi$ with the 
Euclidean vertical segments at the discontinuity points of
$\varphi$  (Theorem \ref{existence-unbounded}).
 
 In  \cite[Theorem 1]{ADR}, Alias, Dajczer and Rosenberg   prove an 
 existence result  for minimal graphs  on  bounded domains in $Nil_3,$  whose 
boundary is  $C^3$ and strictly convex, with continuous 
boundary data.

 We also solve the Dirichlet problem in a half-plane with 
 piecewise continuous (smooth) bounded  boundary   data  
 (Theorem \ref{existence-unbounded}).  All these graphs over the 
half-plane  have 
linear growth (see Definition \ref{growth-def}).

\vskip1.5mm

Thirdly, we provide a geometric construction of a Scherk type surface over 
a triangle and we use it to prove that
any non negative  prescribed continuous boundary data $\varphi$ on 
the boundary  of a wedge of angle $ \theta \in [\pi/2, \pi[$,  has a minimal 
extension $u$ in the wedge with at least quadratic 
growth.   The existence of the 
Scherk type surface and of the graph over a wedge, are  consequences of  more 
general existence results 
(Theorem \ref{Scherk-gen}, \ref{P.wedge-gen}).
 When the domain is a half-plane,  we are  able to construct bounded 
minimal graphs and minimal graphs with quadratic growth, provided some geometric 
hypothesis on the boundary data 
(Corollay \ref{coro-halfplane}).

Our result in $Nil_3$ is in contrast with the classical results for 
the minimal equation in Euclidean space. In $\r^3,$  if the boundary data on a
wedge of angle less that $\pi$ is zero (respectively bounded) then 
the minimal solution is zero (resp. bounded) \cite{R-SaE}.

\vskip1.5mm

Finally,  here  are some  questions that arise from a 
glimpse at  our developments together with the known results in the Euclidean 
space.
\begin{itemize}
\item The following maximum principle holds  for the minimal surface equation in 
$\r^3.$ Let $\Om$ be a strip: if   the boundary    data
$\varphi$ is bounded above by a constant $A$ on $\partial \Om$, then any 
minimal extension in $\Om$ is also bounded above by the same constant $A$ 
\cite[ Theorem 2.2, pg. 168]{R-SaE}.

It is very interesting to check if this {\em maximum principle} holds in 
the Heisenberg space. The solutions given by  Theorem \ref{existence-unbounded}(A)  satisfy such property.

\vskip1.5mm

\item  Let $\Om$ be a strip. In  $\r^3,$  the  solution to the Dirichlet problem  
for  the minimal  surface  equation in $\Om$ is not unique in general. In fact, P. Collin 
\cite{Co}, using the Jenkins-Serrin 
construction \cite{JS}, showed that there exists a smooth unbounded boundary 
value data $\varphi$ such that $\varphi$ admits an infinity of minimal 
extensions to $\Om.$ Any such extension has at most 
linear growth.

\vskip1mm

The   non uniqueness (or uniqueness) of   the minimal extension on a 
strip  in the Heisenberg space is an open problem.

\vskip1mm

On the other hand, by combining the results in  \cite {CK},   
we  have  the following uniqueness result in $\r^3:$  if $\Omega$ is  a strip and if 
$\varphi$ 
is a continuous bounded boundary  data over $\partial \Om$, then there 
exists a unique minimal extension $u$  in $\r^3$
to $\Om$ \cite[Thm 3.3]{CK}.

It is very interesting to investigate if  an analogous  uniqueness result holds 
in the Heisenberg space.
\end{itemize}

\section{$Nil_3(\tau)$  and the minimal surface equation}
\label{coordinates}

\subsection{The Setting}

The three dimensional  Heisenberg group  $Nil_3(\tau)$ can be viewed as  $\r^3$
with the following metric:
\begin{equation*}
  ds_\tau^2 =dx_1^2+dx_2^2+\left(\tau(x_2dx_1-x_1dx_2)+dx_3\right)^2,
\end{equation*}
where $\tau$ is a constant different from zero. If $\tau=0,$ we recover
the Euclidean space. In fact, most of our results hold for $\tau=0.$
We will point out when  our  results are new in the Euclidean 
case, as well.

\smallskip

The isometries of $Nil_3(\tau)$ in this model  are generated by the following
maps (see \cite{FMP} for further details).
\begin{align*}
 \varphi_1(x_1,x_2,x_3) &= (x_1+c,x_2,x_3+\tau c x_2) \\
 \varphi_2(x_1,x_2,x_3) &= (x_1,x_2+c,x_3-\tau c x_1) \\
 \varphi_3(x_1,x_2,x_3) &= (x_1,x_2,x_3+c) \\
 \varphi_4(x_1,x_2,x_3) &= ((\cos\theta)x_1-(\sin\theta)x_2,
(\sin\theta)x_1+(\cos\theta)x_2, x_3) \\
 \varphi_5(x_1,x_2,x_3) &=  (x_1,-x_2, -x_3).
\end{align*}

 Notice that $\varphi_3$ and $\varphi_4$ are translations along the  $x_3$ axis
 and  rotations in the plane $x_1,x_2,$ respectively, hence they are Euclidean
isometries, as well. Notice that $\varphi_5$ is also an Euclidean
isometry.

\smallskip

We can express the isometries of $Nil_3(\tau)$ in a complex form. Let  
$z:=x_1 + ix_2$ on $\R^2,$ then any isometry of $Nil_3(\tau)$ is of  one of the 
following  forms
  \begin{equation}
  \label{iso-complex}
  \Psi_1:= \big(\psi_1(z),\  x_3 +
\tau \text{Im}(\bar z_0 e^{i\theta}z)\big), \ \ \Psi_2:= \big(\psi_2 (z),\  -x_3 +
\tau \text{Im}(\bar z_0 e^{i\theta}\bar z)\big)
\end{equation}
where $\psi_1 (z) = e^{i\theta}z + z_0$
and  $\psi_2(z) = e^{i\theta}\bar z + z_0,$ for some $\theta\in \R$ 
and some $z_0\in \co.$
Notice that both $\psi_1$ and $\psi_2$ are isometries of $\R^2.$

\smallskip

In $Nil_3(\tau),$ we consider the left-invariant orthonormal frame
$(E_1,E_2,E_3)$ defined by

\begin{equation}
E_1=\frac{\partial}{\partial x_1}-\tau x_2\frac{\partial}{\partial x_3}, \
E_2=\frac{\partial}{\partial x_2}+\tau x_1\frac{\partial}{\partial x_3}, \
E_3=\frac{\partial}{\partial x_3}
\end{equation}

and we will write vectors with respect to this frame.

\

 In this work we always assume that the open subsets  
$\Omega \subset \R^2$ considered have properly  embedded boundary. More 
precisely 
for any $p\in \partial\Omega$, there exists an open ball $B\subset\R^2$ 
centered at 
$p$ such that $B\cap \partial\Omega$ is the graph of a $C^k$ function defined 
on some open interval, $k\geq 0$.

\

 Let 
$u: \Omega \rightarrow \R$ be a $C^2$ function.
The graph of the function $x_3=u(x_1,x_2)$ is minimal
if and only if $u$ satisfies the {\em  (vertical) minimal surface equation}

\begin{equation}\label{minimal}
{\mathcal D}_{\tau}(u):=\left(1+(u_2-\tau x_1)^2\right)u_{11}
-2(u_1+\tau x_2)(u_2-\tau x_1)u_{12}+\left(1+(u_1+\tau x_2)^2\right)u_{22}=0
\end{equation}

The previous equation can be written in divergence form, that is:

$$
{\rm div_{\r^2}}\left(\frac{\tau x_2+u_1}{W_u}, \frac{-\tau
x_1+u_2}{W_u}\right)=0
$$
where
\begin{equation}
\label{gradient}
W_u=\sqrt{1+\left(\tau x_2+u_1\right)^2+\left(-\tau x_1+u_2\right)^2}.
\end{equation}

Notice that, as $u$ is $C^2$ on $\Omega$ then, by Morrey's regularity 
result \cite{M1}, $u$ is analytic on $\Omega$.
Moreover the standard interior and boundary  maximum principle hold 
for solutions of equation \eqref{minimal}.
 Let us recall the geometric formulations of the maximum principle, 
 that will be largely applied in this article. One can find a
proof of the maximum principle in \cite{FS}.

\

{\bf Maximum  Principle.}
\begin{enumerate}
\item {\em Let $M_1$ and $M_2$ be two connected minimal surfaces 
immersed in $Nil_3$ and assume $M_2$  is complete.
Let $p$ be an interior point of both of them.  Assume that $M_1$ 
lies on one side of $M_2$ in a neighborhood of $p,$ then, $M_1$ coincides with 
$M_2$ in a neighborhood of $p$  and, by analytic 
continuation,  $M_1\subset M_2.$}

\item  {\em Let $M_1$ and $M_2$  be two compact, connected minimal 
surfaces 
in $Nil_3,$  both with boundary. Assume that 
$M_1 \cap \partial M_2=\emptyset$ 
and that $M_2 \cap \partial M_1=\emptyset$ .

Let $p$ be an interior point of both of them.  Then it cannot occur that 
 $M_1$ lies on 
one side of $M_2$ in a neighborhood of $p$.}

\end{enumerate}

\begin{definition}\label{extension} Let $\Omega$ be an open subset of $\r^2$
and let $\varphi$ be
a function continuous on $\partial\Omega$ except,  maybe, at a finite
number of points
$\{p_1,\dots, p_n\},$ where $\varphi$ has
left and right limit.  We say that $u$ is a minimal extension of $\varphi$ 
over $\bar\Omega$   if

\begin{enumerate}
\item  $u:\bar\Omega\setminus \{p_1,\dots p_n\}\longrightarrow\r,$ is 
continuous, smooth on
$\Omega,$   and satisfies
the minimal surface equation \eqref{minimal}.

\item $u_{|\partial\Omega\setminus \{p_1,\dots p_n\}}=\varphi$.
\end{enumerate}
\end{definition}

\begin{remark}\label{translations}
\begin{enumerate}
\item[]{{\rm (A)}} Let us describe the effect of the isometries of
$Nil_3(\tau)$ on a
curve  $\Gamma$ in the $x_1$-$x_2$ plane. Let $\varphi$ be any
isometry of $Nil_3(\tau).$ The curve $\varphi(\Gamma)$ is not contained in
the $x_1$-$x_2$ plane in general. The projection $\pi(\varphi(\Gamma))$
 of such curve on the $x_1$-$x_2$ plane is obtained from the curve
 $\Gamma$ by an isometry of the Euclidean $x_1$-$x_2$ plane, because
the trace of any isometry of $Nil_3(\tau)$ on the $x_1$-$x_2$ plane
is an isometry of $\r^2.$

This implies, for example, that   the notion of convexity of a curve
$\Gamma$ in the $x_1$-$x_2$ plane  is somewhat intrinsic, for the following
reason.
Assume that $\Gamma$ is convex, then $\pi(\varphi(\Gamma))$
is convex, for any isometry $\varphi$ of $Nil_3(\tau).$

\item[]{{\rm (B)}}
 Let $\Omega$ be a domain in $\R^2$ and let
$v : \Omega\rightarrow \R$ be any function defined on   $\Omega$ with 
boundary value $v_{|\partial\Omega}=\varphi.$ Apply   the  isometry $\Psi_1$ 
defined in \eqref{iso-complex}. 

\begin{align*}
 \Psi_1 \left( \{(z, v(z)),\ z\in \Omega \}\right) &=
 \{ \big(\psi_1 (z),\  v(z) + \tau \text{Im}(\ov z_0 e^{i\theta}z)\big),\ z\in
\Omega \} \\
 &= \{\big(\wt z,\, \wt v (\wt z)\big),\ \wt z \in \psi_1(\Omega )\},
\end{align*}
where $\wt z:= \psi_1 (z)$ and
$\wt v (\wt z):= (v \circ \psi_1^{-1})(\wt z) +
\tau \text{Im}(\ov z_0 e^{i\theta}\psi_1^{-1}(\wt z) )$ for any
$\wt z \in \psi_1 (\Omega )$. Consequently the image of the graph of $v$ by the
isometry $\Psi_1$ is again a graph with boundary value $\wt v$.

\end{enumerate}

\end{remark}

 \subsection{Examples of Complete Minimal Graphs}
\label{examples}
\begin{enumerate}
 \item The graph of a linear entire function $u(x_1,x_2)=ax_1+bx_2+c$,  
$a,b,c\in\r,$ is a minimal surface, that is usually called {\em plane}. Notice 
that a vertical Euclidean plane is also a 
minimal surface and it is flat.

 \item  Let $\phi(t,s)=(r(t)\cos s, r(t)\sin s, h(t))$  a parametrization of a  
 rotationally invariant surface ($\tau=\frac{1}{2}$).
 Looking for minimal solutions one gets either planes ($h(t)=const$) or a 
$1$-parameter 
family of surfaces, {\em vertical catenoids} depending on  a  
parameter $  r_0>0$, given by:
\begin{equation*}
 h(r)=\pm\int^r_{ r_0}
 \frac{r_0\sqrt{s^2+4}}{2\sqrt{s^2- r_0^2}}ds, 
\qquad r\geq  r_0.
\end{equation*}
This means that half of the catenoid is a graph over the exterior 
domain $r\geq  r_0$ 
with zero boundary values \cite{MN}. One finds a detailed study of the  vertical 
catenoids  in  \cite{BC}.

 \item  C.B.  Figueroa, F. Mercuri and R. H. Pedrosa \cite{FMP}, \cite{Fi} 
classified all
the minimal graphs invariant by a one parameter group of left invariant 
isometries.
 Such surfaces are the graphs of  functions  of the following form

 \begin{equation}
 \label{FMP-eq}
 u_a(x_1,x_2)=\tau x_1x_2+a\left[2\tau x_2\sqrt{1+4\tau^2 x_2^2}+
 \sinh^{-1} (2\tau x_2)\right]
 \end{equation}

 \item B. Daniel \cite[Examples 8.4 and 8.5]{D} constructed entire minimal 
 graphs of the form $f(x_1,x_2)=x_1g(x_2)$
 for some real function g with linear growth.
\end{enumerate}

\subsection{Horizontal Catenoids in $Nil_3(\tau)$}
\label{hor-catenoid}

 Let us describe {\em Horizontal Catenoids} in   $Nil_3(\tau).$
 In \cite{DH}, B. Daniel
and L. Hauswirth have constructed a family of horizontal catenoids
${\mathcal C}_\alpha$, $\alpha \in \, ]0, +\infty[\,$,
in $Nil_3(\frac{1}{2}).$

The family ${\mathcal C}_\alpha$ has the following
description \cite[Theorem 5.6]{DH}.
\begin{itemize}
\item The intersection of ${\mathcal C}_{\alpha}$ with any vertical
plane $x_2=c,$ $c\in\r$, is a nonempty, embedded, closed curve,  
convex with respect to the Euclidean metric.

\item The surface ${\mathcal C}_{\alpha}$ is properly embedded.

\item ${\mathcal C}_{\alpha}$ is conformally equivalent to
$\co \setminus \{0\}$.

\item The projection of ${\mathcal C}_{\alpha}$ in the $x_1$-$x_2$ plane
is the following subset of $\r^2$
\begin{equation*}
 \pi({\mathcal C}_{\alpha})=\left\{(x_1,x_2)\in\r^2,  \ |x_1|\leq
 \alpha\cosh \left(\frac{x_2}{\alpha}\right)\right\}
\end{equation*}

\item  The surface ${\mathcal C}_{\alpha}$ is invariant by rotation 
of angle $\pi$ around the  $x_1,$ $x_2$ and  $x_3$ axis.

\end{itemize}

\begin{remark}
\label{gen-catenoids}
We extend the construction of the family ${\mathcal C}_{\alpha}$ in 
$Nil_3(\tau),$   for any value of $\tau >0$.
Consider a new copy of $\R^3$ with coordinates $y=(y_1,y_2,y_3)$ and, for any
real number $\lambda >0$, let $f_\lambda : (\R^3,y)\rightarrow (\R^3, x)$ be
the map: $x=f_\lambda (y) =\lambda y$.

Then, the pullback metric of $(\R^3,x,ds_{1/2}^2)$
on $(\R^3,y)$ induced by $f_\lambda$ is:
\begin{align*}
 g_\lambda &= \lambda^2\left[ dy_1^2 + dy_2^2 +
 \left(\frac{\lambda}{2}(y_2dy_1 - y_1dy_2) + dy_3\right)^2\right] \\
 &= \lambda^2 ds_{\lambda /2}^2.
\end{align*}
Let $X: \Sigma \rightarrow (\R^3,x,ds_{1/2}^2)$ be a conformal and minimal
immersion, where $\Sigma$ is a Riemann surface. We deduce  from the previous
observations that
$Y:= f_{2\tau}^{-1} \circ X :\Sigma \rightarrow (\R^3,y,ds_{\tau}^2)$ is also a
conformal and minimal immersion.

\smallskip

Consequently, for any $\tau >0$, and for any $\alpha >0$, the
surface
${\mathcal C}^{\tau}_{\alpha}:=\frac{1}{2\tau} {\mathcal C}_{\alpha}$ is 
a horizontal catenoid in
$(\R^3,ds_{\tau}^2)$, that is in $Nil_3(\tau)$.

\smallskip

We deduce from the construction that for any $\tau >0$, the family of
horizontal
catenoids ${\mathcal C}^{\tau}_{\alpha}$ of $Nil_3(\tau)$ has a
geometric description analogous to that of
the case $\tau = 1/2$. In particular the family of
the projections $\{ \pi ({\mathcal C}^{\tau}_{\alpha}),\ \alpha >0\}$ is the same for
any $\tau>0$ since for any $\alpha, \beta >0$, we can pass from
$\pi ({\mathcal C}_{\alpha})$ to $\pi ({\mathcal C}_{\beta})$ by a suitable
homothety.
\end{remark}

\section{Non existence Results  in $Nil_3(\tau)$}
\label{Dirichlet}

As, in the following, we will deal with convex and non convex curves, 
let us recall some properties of them.
Let  $\Gamma$  be a Jordan curve in the $x_1$-$x_2$ plane
and let $\Omega$ be the open  bounded subset of the
$x_1$-$x_2$ plane such that $\partial\Omega=\Gamma$.  As the projection
of $Nil_3(\tau)$ on the first two
coordinates is a Riemannian submersion on $\r^2,$ it makes sense to assume
that $\Gamma$ is convex as an Euclidean curve.
Recall that $\Gamma$ is convex if and only if, for any point $p\in\Gamma,$
there exists a straight line $l_p$ passing through $p$ such that
$\Omega\subset \r^2\setminus l_p.$ Recall, moreover, that this definition is
equivalent to say that $\bar\Omega$ is convex, that is: for any two points 
$p,\q\in\bar\Omega,$ the segment between $p$
and $q$ is contained in $\bar \Omega.$
We observe that,  if $\Gamma$ is
not convex, then either

\

\begin{itemize}
\item[]{($\star$)}\label{convex1} There exists a point $p\in \Gamma,$ a
straight
line $L$ passing through $p$ and a neighborhood $V$ of $p$ in $\r^2,$ such that
$(L\cap V)\setminus p\subset\Omega$ (see Figure \ref{non-convex-four}),

or

\item[]{($\star\star$)}\label{convex2} There exists a closed arc
$\wt \gamma \subset \Gamma$ with endpoints $\wt p, \wt q$, and there exists a
segment $\hat l $ such that:
\begin{itemize}
 \item the straight line $L$ containing $\hat l$ is parallel (and distinct) to
the straight line passing by $\wt p$ and $\wt q$, 
  (therefore $\wt p, \wt q \not\in L$),

\item $\wt \gamma \cap L$  is infinite,

\item $\wt \gamma$ remains in one closed side of $L$,

\item Let $p\in \hat l$ (resp. $q\in \hat l$ ) be the first intersection point
of $\wt \gamma$ with $\hat l$, coming from $\wt p$ (resp. $\wt q$). Denote by
$l$ the segment $[p,q]$.

Then $p$ and $q$ are interior points of $\hat l$ and
$\hat l \setminus l \subset\Omega$ (see Figure \ref{non-convex-two}).
\end{itemize}

\end{itemize}

\begin{figure}[!h]
\centerline{
\subfigure[Non convex  $(*)$]{
\includegraphics[scale=0.3]{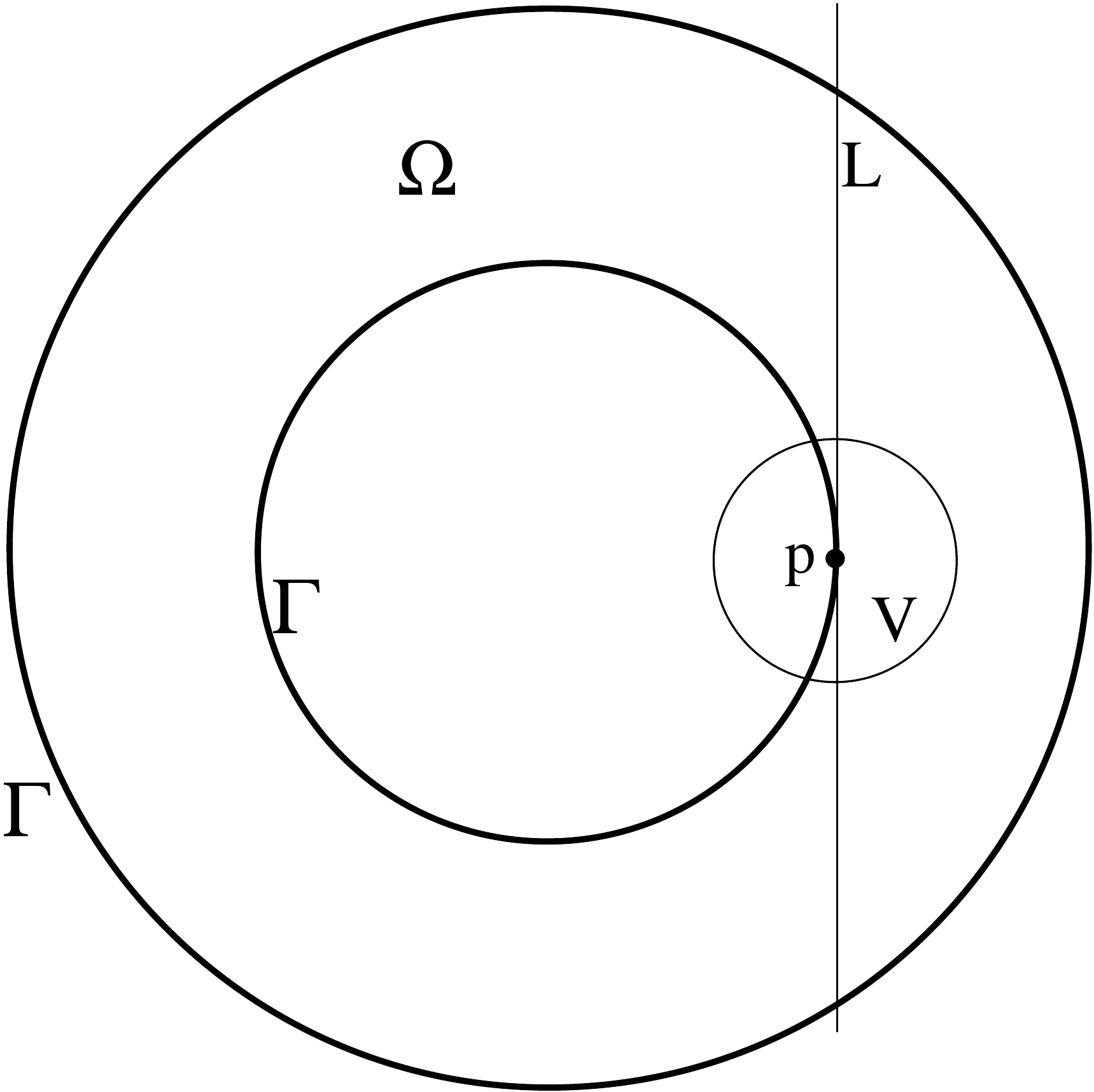}
\label{non-convex-four}
} \hskip18mm
\subfigure[Non convex $(**)$]{
\includegraphics[scale=0.3]{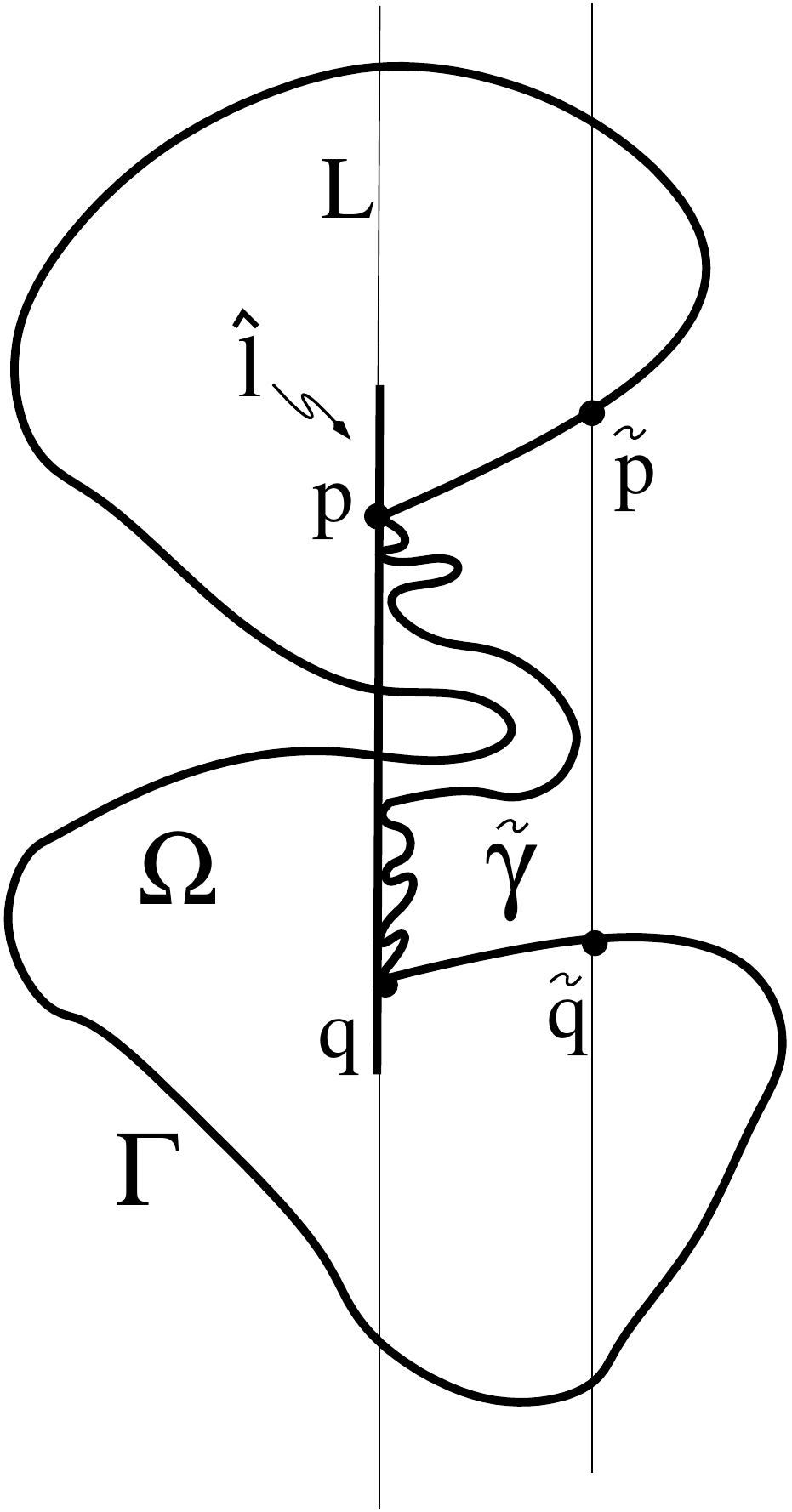}
\label{non-convex-two}
}}
\caption{Non Convexity}\label{}
\end{figure}

 Our first result is a non existence theorem on bounded domains.

\begin{theo}\label{non-existence}
Let  $\Omega$ be a bounded domain  such that $\Gamma=\partial\Omega$ is a
non convex $C^k$ curve, $k\geq 0.$  Then there exists a
$C^k$ function
$\varphi$ on $\Gamma$ that does not admit a  minimal extension over
$\bar\Omega$
in $Nil_3(\tau).$
\end{theo}

\begin{proof}
 We do the proof for $\tau=\frac{1}{2}$. It will be clear that the proof 
 is analogous for any $\tau>0$ (by Remark \ref{gen-catenoids}).
We will do the proof  for $\varphi$ of class $C^0.$
By Remark \ref{translations} (B), it is enough to prove the result for
$\psi(\Gamma),$
where $\psi$ is any  Euclidean isometry of $\r^2$ and for some function
$\tilde\varphi$ over  $\psi(\Gamma)$.
The desired function $\varphi$ will be obtained by modify $\tilde\varphi$ 
accordingly.

\

We assume that $\Gamma$ satisfies ($\star\star$)  and we use the same
notations as there.
The proof in the case   of $\Gamma$ satisfying ($\star$)
 follows easily.

 \

 Without loss of generality, we can assume that $l$ is parallel to
 the $x_2$ axis and    that the $x_1$ axis intersects $l$ at its middle point
 $p_0$ of coordinates $(d,0).$
 Moreover  we assume that  each  point of  $l\cap \wt \gamma$ is
 a local minimum of $\wt \gamma$ for the coordinate $x_1,$ since the case where
 each  point of  $l\cap \wt \gamma$  is  a local maximum for
the  coordinate $x_1$ can be handed analogously (see Figure \ref{picture1}).

\

\

 \begin{figure}[!h]
 \centerline{\includegraphics[scale=0.4]{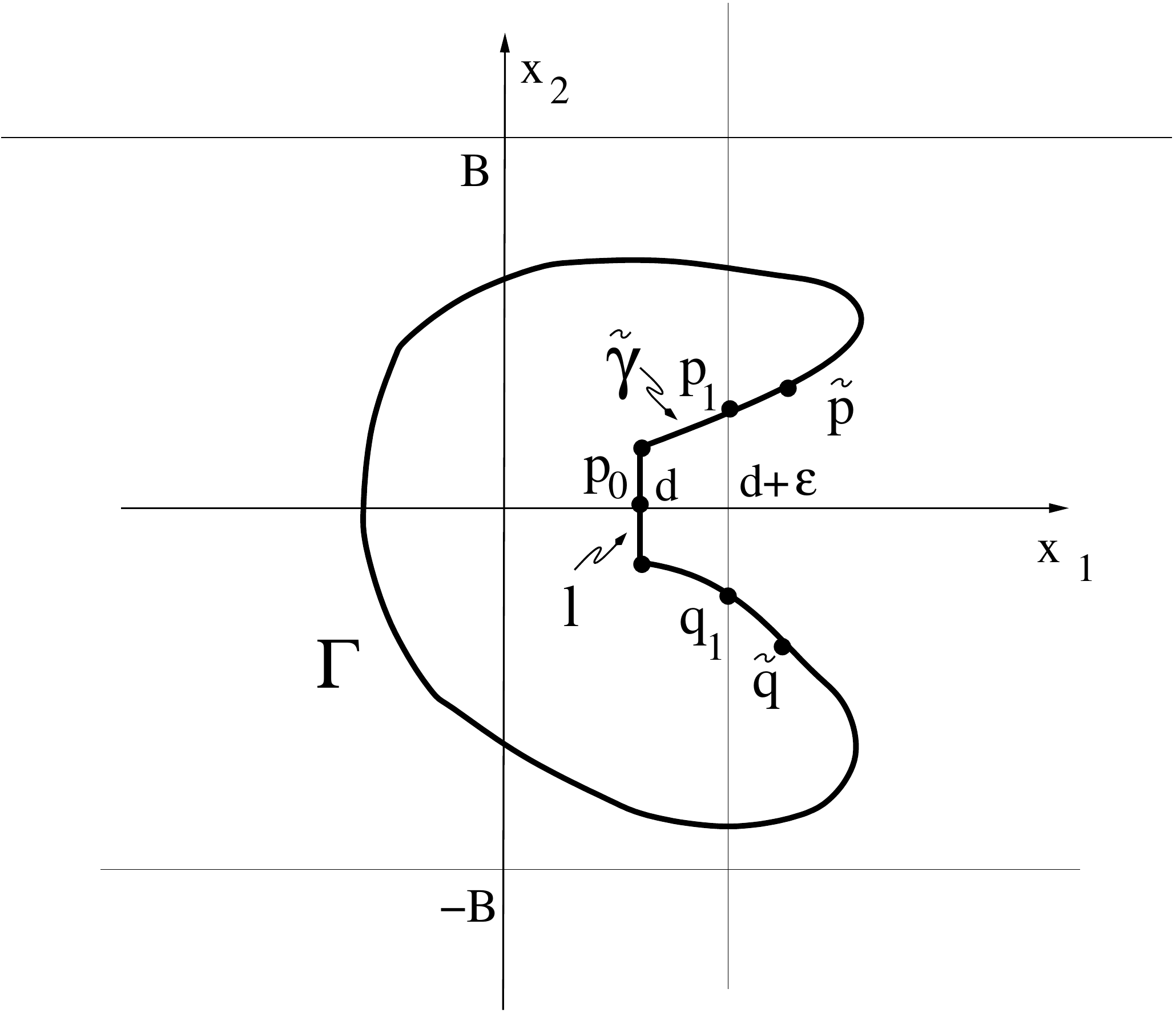}}
 \caption{ The curve  $\Gamma$}
 \label{picture1}
 \end{figure}

We will consider the continuous 1-parameter   family of catenoids
${\mathcal C}_{\alpha}$, $\alpha >0$,  described in
Section \ref{hor-catenoid}. Recall that the
projection in the $x_1$-$x_2$ plane
is the following subset of $\r^2$:
\begin{equation*}
 \left\{(x_1,x_2)\in\r^2,  \ |x_1|\leq \alpha\cosh
\left(\frac{x_2}{\alpha}\right)\right\},  \alpha>0.
\end{equation*}
Notice that the projection $\pi({\mathcal C}_{\alpha})$ is  equal to the
projection
of an Euclidean catenoid with axis $x_2$ on the $x_1$-$x_2$ plane. Therefore,
the projections are obtained one from
the other by an homothety  and, next to the waist, they become flatter as $\alpha$ increases.

\medskip

 Let $\varepsilon>0$ and let $p_1\in \wt \gamma$
(resp. $q_1 \in \wt \gamma$), be the first point on $\wt \gamma$ with $x_1$
coordinate equals to $d+\varepsilon$, coming from $p$ (resp. $q$) and going to
$\wt p$ (resp. $\wt q$). Such points exist if $\varepsilon$ is small 
enough.
We denote by $\gamma_1$ the sub-arc of $\wt \gamma$ with
endpoints $p_1$ and $q_1$. 

\smallskip

As $(\star\star)$ holds, the points of $\gamma_1\setminus l$
have $x_1$ coordinate strictly greater that $d.$

\smallskip

 Let $B$ a positive constant such that
 $\Gamma\subset \{(x_1,x_2)\in \R^2 ,  \ |x_2|\leq  \frac{B}{2}\}.$

Consider the unique catenoid ${\mathcal C}_{\mu}$ such that
\begin{equation}
\mu+\frac{\varepsilon}{4}=\mu\cosh \frac{B}{\mu}.
\end{equation}
 This means that the portion of $\pi ({\mathcal C}_{\mu})$ defined by:
\begin{equation*}
 \pi ({\mathcal C}_{\mu}) \cap
\{(x_1,x_2)\in \R^2, \ x_1 \geq 0,\ 0\leq x_2 \leq B\},
\end{equation*}
is contained in the following vertical strip of $\R^2$ of
width $\varepsilon /4$:
\begin{equation*}
 \{(x_1,x_2) \in \R^2,\ \mu \leq x_1 \leq \mu +\frac{\varepsilon}{4}\}.
\end{equation*}
Up  to  translate $\Gamma$ along the $x_1$ axis, we can assume that the $x_1$
coordinate of $p_0$ is $d=\mu+\frac{\varepsilon}{4}$.  This
choice guarantees that
$\pi({\mathcal C}_{\mu})$  does not intersect $\gamma_1.$

Then, deform $\pi({\mathcal C}_{\alpha})$  by an homothety from the origin,
$\alpha$ going from $\mu$ to $\mu + \varepsilon /4$,
 in order to find a first contact point with the curve
$\gamma_1$.
Let $\mu'\in (\mu, \mu + \varepsilon/4)$ such that $\pi(C_{\mu'})$ 
is the desired homothetic image 
(see Figure \ref{picture2}).

\begin{figure}[!h]
 \centerline{\includegraphics[scale =0.4]{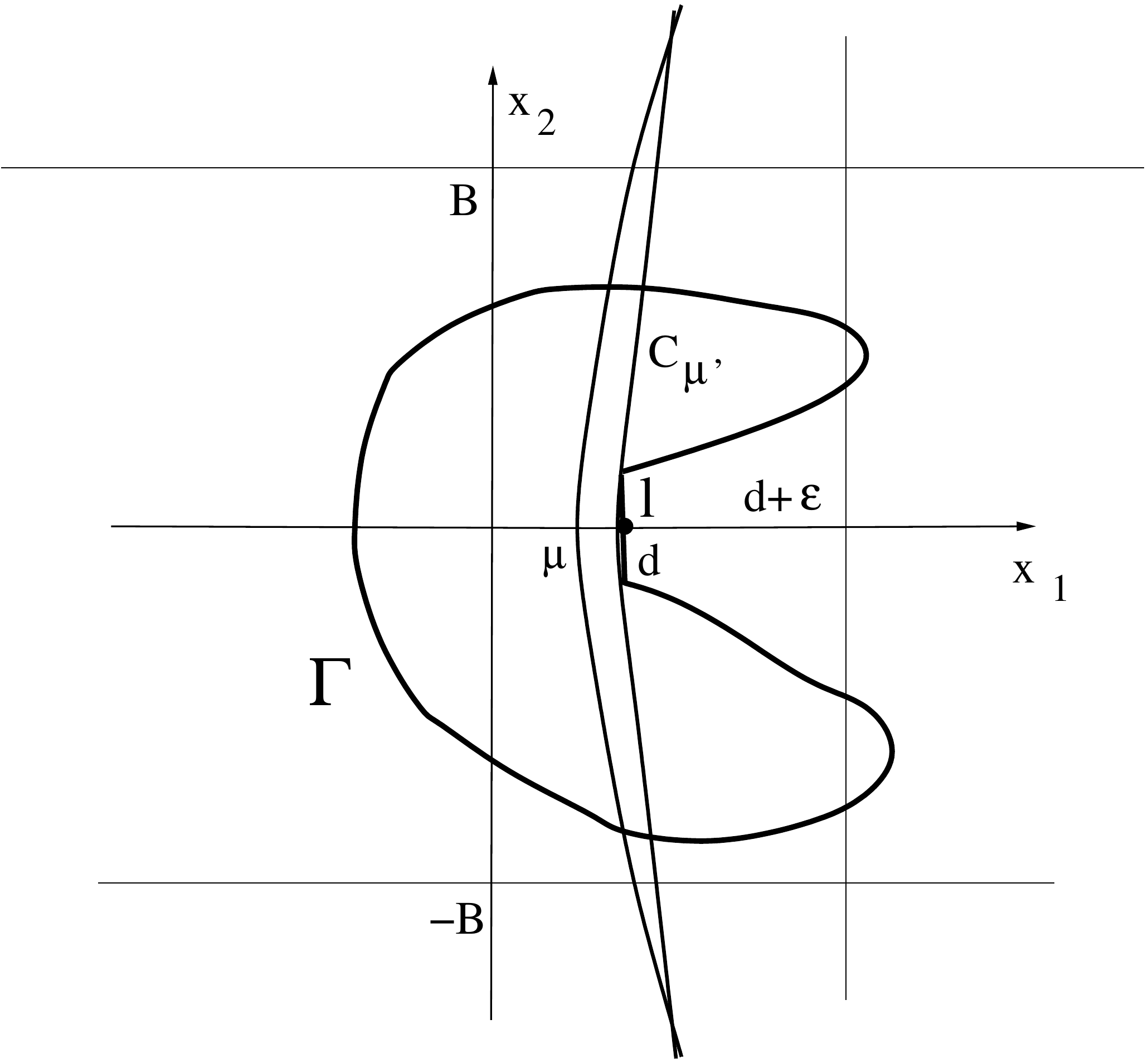}}
 \caption[]{ The deformation of ${\mathcal C}_{\mu}$ }
 \label{picture2}
\end{figure}

\

By our choice of $\mu$ and $d$ and because of the geometry of  the curves    
$\pi({\mathcal C}_{\alpha}),$ we can assume that the contact points between
$\pi({\mathcal C}_{\mu'})$ and $\gamma_1$ are interior points of $\gamma_1.$
Denote by $q_0$ one of the contact points and   notice that
 $q_0$ does
not belong to the interior of $l$.
 Denote by $ \gamma_0$ the sub-arc of $\gamma_1$
whose endpoints are $q_0$ and one of the endpoints of $l$, and such that
$\gamma_0 \cap l$ is infinite.

Let $D >0$ be such that, for any $\alpha\in[\mu,\mu']$ and any 
$p\in{\mathcal C}_{\alpha}\cap \{(x_1,x_2,x_3)\in\R^3, \ |x_2|\leq B\}$, we 
have
$|x_3(p)|\leq  D.$

Now define  a continuous, non negative  function $\tilde\varphi$ on $\Gamma$ such that
\begin{equation*}
 \tilde\varphi=\left\{\begin{array}{l}
0 \  \ {\rm on}\ \Gamma\setminus\tilde\gamma\\
3D \  \ {\rm on}\ \gamma_1\\
\end{array}\right.
\end{equation*}
Assume that there exists a minimal extension $u$ of $\tilde\varphi$
over $\bar \Omega.$
As the function $u$ is continuous in $\bar\Omega,$   for any $\eta_1>0$ 
there exists $\eta_2>0$
such that
\begin{equation}
\vert u(\wh p)-3D\vert\leq \eta_1, \ \ \ {\rm for \ any}\
\wh p=(x_1,x_2)\in \ov \Omega,  \  x_1(q_0)-\eta_2\leq x_1\leq x_1(q_0),  \ 
x_2=x_2(q_0)
\end{equation}
where $x_i(q_0)$ is the $x_i$ coordinate of $q_0$, $i=1,2$.

This means that $u$ have a small variation on a small
segment, $\delta,$ parallel to the $x_1$ axis,  $\delta$ starting at $q_0$
and having $x_1$ coordinates less than    $x_1(q_0)$.

Consider a  homothety  $\pi({\mathcal C}_{\mu''})$ of
$\pi({\mathcal C}_{\mu'})$ such that $\mu''\in(\mu,\mu')$ is very close to
$\mu'.$  As ${\mathcal C}_{\mu'}$ is the first catenoid
of the family, such that its projection touches $\gamma_1$, we 
have  
$\pi({\mathcal C}_{\mu''})\cap\gamma_1=\emptyset$ and
$\pi({\mathcal C}_{\mu''})\cap\delta\not=\emptyset.$

Now, translate vertically (along the $x_3$ axis) upward ${\mathcal C}_{\mu''}$
such that the translation of ${\mathcal C}_{\mu''}$ does not touch the graph of
$u.$
Then translate vertically downward till the first contact point between the
translation of ${\mathcal C}_{\mu''}$ and the graph of $u$ occurs.
We observe that our construction yields that the first contact point 
between the translation of ${\mathcal C}_{\mu''}$
and the graph of $u$

\begin{itemize}
\item occurs before that ${\mathcal C}_{\mu''}$ touches $x_3=0,$
\item does not occurs at a boundary point  of the graph of $u.$

\end{itemize}

Then, the first contact point between the translation of ${\mathcal C}_{\mu''}$
and the graph of $u$ is an interior point of both surfaces.
This is a contradiction by the maximum principle.
\end{proof}

 In the  next theorem, we extend the result of Theorem \ref{non-existence} 
 to the case of  unbounded domains.

\begin{theo}
\label{non-existence-unbounded}
Let $\Omega$ be an unbounded  and non convex domain.
Assume that the boundary $\Gamma=\partial \Omega$  is composed by a finite 
number 
of connected components, each one being a properly embedded (possibly compact)  
$C^k$ curve, $k\geq 0.$ Then
there exists a
$C^k$ function
$\varphi$ on $\Gamma$ that does not admit a minimal extension over  $\bar\Omega.$
\end{theo}

\begin{proof}
As $\Omega$ is non convex, there exist points $p,q\in\Gamma$ such that the 
segment $I$ joining $p$ and $q$ is not entirely contained in
$\bar\Omega.$  Let $J$ be a connected component of the complement of
$I\cap\bar\Omega$ in $I.$ Since $\Omega$ is connected and $J$ is an open 
segment contained in the complement of $\bar\Omega,$ the  
endpoints of $J$ belong to the same connected component of $\Gamma.$
From now on, the proof is analogous to the proof of Theorem   
\ref{non-existence}, 
taking as $\tilde p,$ $\tilde q,$  the endpoints of the segment $J.$
\end{proof}

\begin{remark} We notice that  our  proofs of Theorem \ref{non-existence}
and  \ref{non-existence-unbounded} hold in $\r^3$ ($\tau=0$), as well.
In the case of bounded domain, an almost analogous result is
stated by H. Jenkins and J. Serrin in
\cite{JS1}, page 185. The result in $\r^3$  for unbounded domain
was unknown, to the best of our knowledge.
\end{remark}

\section{Existence Results in $Nil_3(\tau)$}
\label{S.existence}

\subsection{Compactness Theorem}
\label{compactness-theo}

In the proof of  existence results either on unbounded domain  or
with infinite boundary data,  we will use
strongly
the {\em Compactness Theorem for minimal graphs in $Nil_3(\tau)$}.
Despite of the fact that it is a classical result for the  minimal surface
equation
in several ambient spaces, we
 clarify  which are the main ingredients of the proof in $Nil_3(\tau)$.

 The
Compactness
Theorem yields that any $C^{2,\alpha}$ bounded sequence $(u_n)$  of solutions of the
minimal surface equation on a domain $\Omega$ in $\r^2$
 admits a subsequence that converges uniformly in the
$C^2$ topology, on any compact subset of  $\Omega$ to a solution of the minimal
surface equation.
The  proof of this  result relies on  Ascoli-Arzel\'a Theorem, so one need to prove that
the sequence $(u_n)$ is uniformly bounded in the $C^{2,\alpha}$ topology. By
Schauder theory, the $C^{2,\alpha}$ {\em
a-priori} estimates follow  from the $C^{1,\beta}$  {\em a-priori} estimates.
These last estimates follow, by Ladyzhenskaya-Ural'ceva theory, from $C^1$
{\em
a-priori} estimates.  By  \cite[Theorem 3.6]{RST},  such estimates
follow from uniform height estimates. As  \cite[Theorem 3.6]{RST} 
is stated in a
more general  situation, we state it on our case.

\

{\bf Theorem.} [ \cite[Theorem 3.6]{RST}]  {\em Let
$\pi:Nil_3(\tau)\longrightarrow
\r^2$
be the Riemannian submersion of Heisenberg space on $\r^2.$ Let
$\Omega\subset\r^2$ be a relatively compact domain
and $u:\Omega\longrightarrow\r$ be a $C^2$ section of the submersion
$($i.e. a  vertical graph$)$ satisfying \eqref{minimal}. 
 Assume that $C_1$ is a positive constant such that 
$\lvert u\lvert <C_1$ on $\Omega$. 
Then, for any positive
constant  $C_2$, there exists a constant
$\alpha=\alpha(C_1,C_2,\Omega)$ such that  for any $p\in\Omega$ with
$d(p,\partial\Omega)\geq C_2$,  we have
\begin{equation*}
W_u(p)<\alpha .
\end{equation*}
}

\

The uniform bound on $W_u,$ defined as in \eqref{gradient},
clearly gives  $C^1$ a-priori estimates on any compact subset of 
$\Omega$.

\begin{remark}\label{height}
Let $(u_n)$ be  a sequence of $C^2$ functions satisfying the minimal 
surface equation on a domain $\Omega.$  We deduce from above that, once we have 
uniform height estimates for $(u_n),$ there exists 
a subsequence  of $(u_n)$ that converges $C^2$ on any compact subset 
of $\Omega$ to a solution of the minimal surface equation.
\end{remark}

\subsection{Construction of Barriers}
\label{barrier}
In  the proof of existence results, in order to prove that the solutions 
takes the given boundary value,  it is important to get
barriers at a convex point of the boundary, where the boundary data is
continuous.  Our construction of barriers is strongly inspired by the 
analogous construction in $\h^2\times\r,$ by the second and the third author
in \cite{ST}.
Let us first define a convex point of a domain $\Omega.$

\begin{definition}
Let
$\Omega \subset \R^2$ be a domain. We say that a boundary point
$p\in \partial \Omega$ is a {\em convex point of $\Omega$} if there exist an
open neighborhood $V$ of $p$ in $\R^2$ and a straight line $L\subset \R^2$
passing through $p$  such that:
\begin{itemize}
 \item $V\cap \Omega$ stays in one side  of $L$,

 \item $(V\cap L) \cap \Omega = \emptyset$.
\end{itemize}
\end{definition}

Let $\Omega$ be a  (not necessarily bounded)  domain and let 
$\varphi:\partial\Omega\longrightarrow \R$ be a function.
We recall what is a barrier at a point $p_0\in\Gamma:=\partial\Omega$ 
with respect to the function $\varphi.$

Let $p_0$ be a  point of $\Gamma.$  One says that
$p_0$ admits un  {\it upper (lower) barrier} with respect to $\varphi$ 
if the following holds.
For any positive $M$ and any $k\in{\mathbb N},$ there exists an
open set $V_k$ containing $p_0$ in its boundary and a function
$\omega^+_k$ (resp. $\omega^-_k$)  in
$C^2(V_k\cap \Omega)\cap C^0(\overline{V_k\cap \Omega})$ such that

\begin{enumerate}
\item $\omega^+_k(q)_{|\partial \Omega\cap \bar{V_k}}\geq \varphi(q),$
$ \omega^+_k(q)_{|\Omega\cap \partial V_k}\geq M$
(resp. $\omega^-_k(q)_{|\partial \Omega\cap \bar{V_k}}\leq \varphi(q)$,
$ \omega^-_k(q)_{|\Omega\cap \partial V_k}\leq -M$).

\item ${\mathcal D}(\omega^+_k)\leq 0$
(resp. ${\mathcal D}(\omega^-_k)\geq 0$) where $\mathcal D$ is defined in
\eqref{minimal} .

\item $\omega^+_k(p_0)=\varphi(p_0)+\frac{1}{k}$
(resp.  $\omega^-_k(p_0)=\varphi(p_0)-\frac{1}{k}$).

\end{enumerate}

Now, let $p_0\in \Gamma$ be a convex point of $\Omega$ and 
assume that  $\varphi$ is continuous at $p_0$.
 Let $M$ be any positive real number.
We show how to construct an upper  barrier at the  point $p_0$  
(see Figure \ref{picture3}).

 \begin{figure}[h]
 \centerline{\includegraphics[scale =0.4]{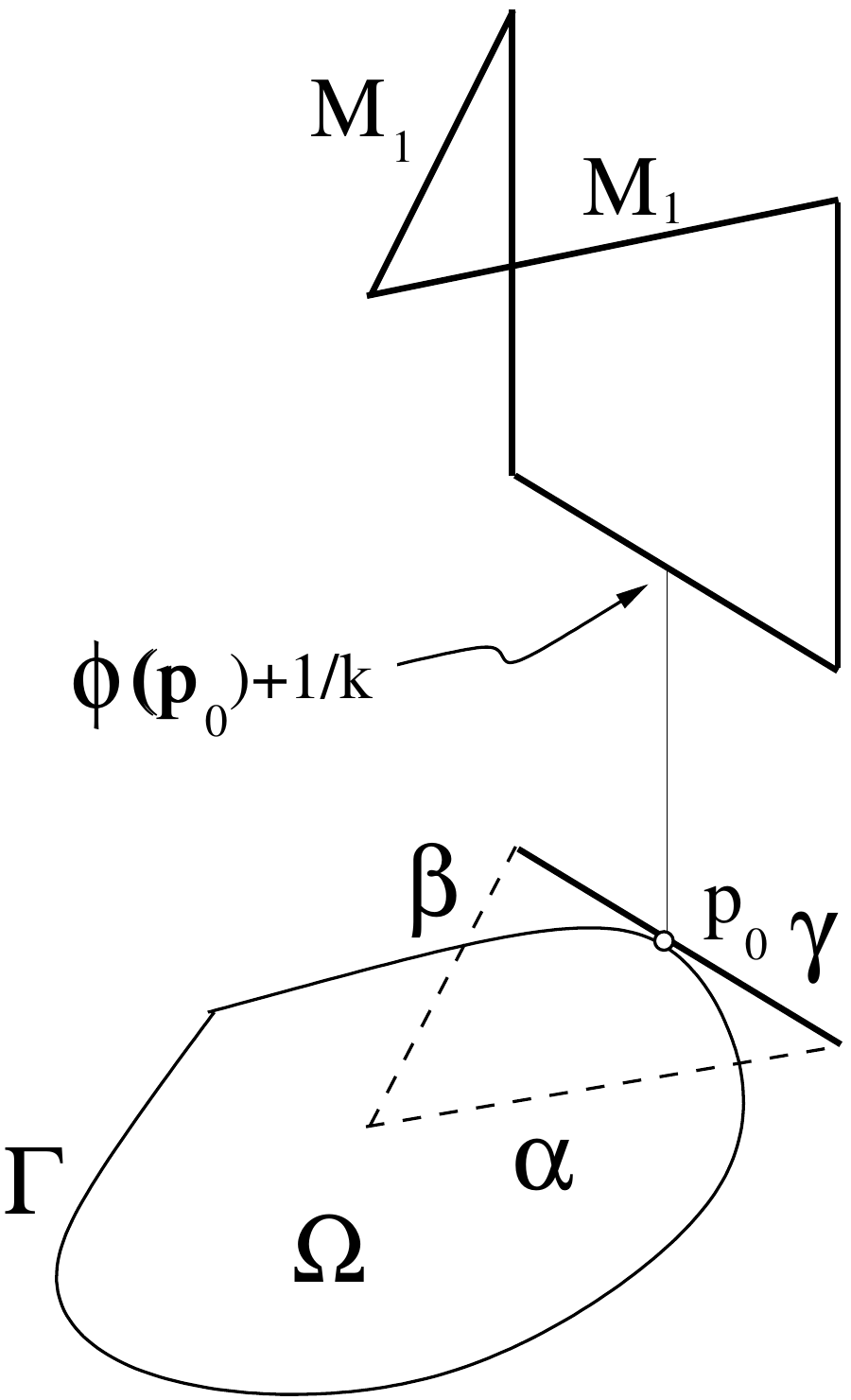}}
 \caption[]{ The upper barrier }
 \label{picture3}
 \end{figure}

 Consider a triangle  $T$ in $\r^2$ with sides $\alpha,$
 $\beta,$ $\gamma.$ Let $A,B,C$ be the vertices of $T$ 
 labeled such that $A,B,C$ are  opposite to  $\alpha,$ $\beta,$ $\gamma$ 
respectively.
Let $M_1$ be a positive constant  
 to be chosen later and  let $\tilde A,$
($\tilde B,$ $\tilde C$) be the endpoints of the vertical segment 
above $A$ ($B,$  $C$ respectively) of  Euclidean lenght $M_1$. Let 
$\tilde \alpha$ ($\tilde \beta$ respectively) be the  segment  
projecting one to one
on $\alpha$
($\beta$ respectively) with endpoints  $\tilde B,$  $\tilde C$ ($\tilde A,$
$\tilde C$ respectively) and with third coordinate  equal to $M_1.$
Denote by $L_B$ ($L_A$)  the vertical segment  between $B$ and $\tilde B$
($A$ and $\tilde A$ respectively). We solve the Plateau problem with boundary
$\gamma,$ $L_B,$ $L_A,$ $\tilde \alpha,$
$\tilde\beta$ (see \cite{M2}).   The solution of the Plateau problem   
is a graph in the interior
of $T$ of a function $v$ having  value zero   on the interior of $\gamma$ and
value  $M_1$ on $\alpha\cup\beta$ (see the
proof of  \cite[Theorem 1]{ADR}).

As $\varphi$ is continuous,  there exists $\varepsilon>0$ such that, for
any $q\in\partial\Omega$ such that $d(p_0, q)<\varepsilon,$ one has
$\varphi(q)<\varphi(p_0)+\frac{1}{k}.$

In the definition of barrier, let $V_k$ be  the triangle $T$ 
defined  above, such that $\gamma$ touches
$\Gamma$ at $p_0,$ $\gamma\cap \Omega=\emptyset$  and
$T\cap\Omega\not=\emptyset.$
Moreover, choose $T$ such that for any $q\in T\cap \partial \Omega,$
one has $d(q,p_0)<\varepsilon.$

 We choose $M_1$ such that  
$M_1>\max(M, \varphi(p_0)+\frac{1}{k})$ and we set
$\omega^+_k=v+\varphi(p)+\frac{1}{k}$,
where $v$ is the function on  $T$
described   above.

Hence the function $\omega^+_k$ satisfies the required properties.

\smallskip

The construction of the lower barrier is analogous.

\subsection{Existence on Bounded Domains}
\label{existence-bounded-section}

We state an existence result and we prove it  by using
classical tools of minimal surfaces theory.  See 
\cite[Corollary 4.1, page  325]{ST} for analogous results 
in $\h^2\times\r.$
 In  \cite[Theorem 1]{ADR} or
\cite[Theorem 1.1]{P}     one finds existence results in $Nil_3,$ 
with different regularity assumptions.

\begin{theo}\label{existence}
 Let $\Omega\subset\r^2$   be a bounded convex domain.
Let $\varphi$  be a   continuous
function  on  $\Gamma=\partial \Omega$ except, possibly, at a finite
number of points $\{p_1,\dots,p_n\}$, where $\varphi$ has left and right limits.

 Then, there exists a unique
minimal extension $u$ of $\varphi$ over $\bar\Omega$.
Moreover, the boundary of the graph of $u$ in $Nil_3(\tau)$
is the Jordan curve $\gamma$ given by the graph of $\varphi$ over
$\Gamma\setminus \{p_1,\dots,p_n \}$ and the vertical segments between the left
and right limit of $\varphi$ at any $p_1,\dots ,p_n$.
\end{theo}

\begin{proof} 
We first prove the existence part.
  Let $S$ be a solution of the Plateau problem  for the
Jordan curve $\gamma$  defined in the statement (see \cite{M2}).   As $\Gamma$
is convex, we can compare $S$ with the vertical planes and  by the maximum
principle, we get that the surface $S$ is contained in the   Euclidean cylinder over
$\bar \Omega$.

Then, $S\cap (\Omega \times \R)$ is a graph,  as it is proved in \cite{ADR}.
 The uniqueness,  in the case of continuous boundary data, follows by a
 standard up and down argument  and  the maximum principle.
When $\varphi$ has a finite number of discontinuity points, one uses
the classical argument of H. Jenkins and J. Serrin \cite{JS} adapted
to the metric of $Nil_3,$ by  A. L. Pinheiro  \cite[Theorem 2.1]{P}. 
\end{proof}

\begin{remark}
The result analogous to Theorem \ref{existence} in $\h^2 \times \R$ and
$\widetilde{PSL_2(\r)}$ are consequences of  the results
in \cite{NR}, \cite{NR1}, \cite{ST}  and \cite{Y} respectively. 
The proof of Theorem \ref{existence}
holds in such spaces as well.
\end{remark}

 \subsection{Scherk's Type Surface}
\label{Scherk-section}

 \begin{theo}\label{Scherk-gen}
 Let $\Omega$ a bounded, convex domain such that $\partial \Omega=C\cup\gamma,$ 
where $C$ 
is a convex curve and $\gamma$ is a segment. Let $\varphi : C \rightarrow \R$ be 
a continuous function.
Then, $\varphi$ has a unique minimal extension  over $\Omega$ assuming the value $+\infty$
on the interior of $\gamma$.

 More precisely, there exists a unique continuous function
 $u : \Omega\setminus \gamma \rightarrow \R$ such that:
\begin{itemize}
 \item $u$ is $C^2$ on the interior of $\Omega$ and verifies the minimal surface
equation $(\ref{minimal})$,

\item $u(p_n) \to + \infty$ for any sequence $(p_n)$ of
{\rm int}$(\Omega)$ converging to an interior point of $\gamma$.

\item $u(p)=\varphi(p)$ for any $p\in C\setminus\partial C.$

\end{itemize}
\end{theo}

As a  Corollary of  Theorem \ref{Scherk-gen} we get the following existence 
result.
Let $T\subset \R^2$  be a triangle with sides $\alpha,$ $\beta,$ $\gamma.$ Let  
$A, B, C$ the vertices of $T$ 
labeled such that  they are opposite to
 $\alpha, \beta, \gamma$ respectively.

\begin{coro}\label{Scherk}
 Let $\varphi : \alpha \cup \beta \rightarrow \R$ be a continuous function.
Then, $\varphi$ has a unique minimal extension  over $T$ assuming the value $+\infty$
on the interior of $\gamma$.

 More precisely, there exists a unique continuous function
 $u : T\setminus \gamma \rightarrow \R$ such that:
\begin{itemize}
 \item $u$ is $C^2$ on the interior of $T$ and verifies the minimal surface
equation $(\ref{minimal})$,

\item $u(p_n) \to + \infty$ for any sequence $(p_n)$ of
{\rm int}$(T)$ converging to an interior point of $\gamma$.

\item $u(p)=\varphi(p)$ for any $p\in \alpha \cup \beta$, $p\not= A, B$.
\end{itemize}
\end{coro}

\

{\em proof of Theorem  \ref{Scherk-gen}.}  
  We do the proof for $\tau=\frac{1}{2}$. It will be clear that the 
proof  is analogous for any $\tau>0$ (by Remark \ref{gen-catenoids}).

By Remark \ref{translations}(B), 
it is enough to prove the result for $\psi (\Omega)$ where $\psi$ is any
Euclidean isometry of $\R^2$ and for any continuous function on
$\psi (C)$.

\medskip
  It follows that we can assume
that the segment $\gamma$ is orthogonal to the $x_1$ axis, with:
$0< x_1 (C) < x_1 (\gamma)$.

\smallskip

Let $n\in \n$ and let $\varphi_n$ be the piecewise continuous 
function on $\partial\Omega$
with value $\varphi$ on $C$ and $n$ on ${\rm int}(\gamma)$. 
Theorem \ref{existence} insures the existence and uniqueness of a minimal
extension $u_n$ of $\varphi_n$ to $\Omega$. Namely, for any $n\in \n$:
\begin{itemize}
 \item $u_n$ is continuous on $\Omega\setminus\partial C$,

\item $u_n$ is $C^2$ on int$(\Omega)$ and satisfies the  minimal surface
equation $(\ref{minimal})$,

\item $u_{n}(p) =\varphi_n(p)$ for any $p\in C \setminus \partial C$.
\end{itemize}

\smallskip

\noindent {\bf Claim 1.} {\em There is a subsequence of $(u_n)$
converging to a $C^2$ function $u$ on int$(\Omega),$ satisfying the minimal 
surface equation.}

By Remark \ref{height}, it is enough
to prove that for any compact subset $K$ of int$(\Omega)$, 
there exist uniform height estimates for  the sequence $(u_n).$

We show uniform height estimates  for $(u_n)$ from above  on $K$. A similar reasoning   shows that
uniform height estimates from below hold as well.

\smallskip

Let $d_K>0$ be the Euclidean distance between $K$ and $\gamma$, and let
$B$ be a positive constant such that
$\Omega\subset \{ (x_1,x_2),\ \vert x_2\vert < B/2\}$.

As in the proof of Theorem \ref{non-existence}, we consider the family of
horizontal catenoids ${\mathcal C}_{t}$ of $Nil_3$, $t\in\R_+$.
Recall that the projection of  ${\mathcal C}_t,$  $t>0,$ on the
$x_1$-$x_2$ plane is given by:
\begin{equation*}
 \pi({\mathcal C}_{t})=
 \left\{(x_1,x_2)\in\r^2,  \ |x_1|\leq t\cosh
\left(\frac{x_2}{t}\right)\right\}.
\end{equation*}
We set $L^+_B (t)=\partial\pi({\mathcal C}_{t}) \cap 
\{(x_1,x_2)\in \R^2, \ x_1 >0\,, \vert x_2 \vert \leq B\}$.
Then:
\begin{align*}
 L^+_B (t) \cap \{(x_1,B),\ x_1 \in \R\} &=
\{ (t\cosh \left(\frac{B}{t}\right), B)\} \\
 L^+_B (t) \cap \{(x_1,0),\ x_1 \in \R\} &=
 \{(t,0)\}.
\end{align*}

Then we have
\begin{equation*}
  L^+_B (t) \subset \{(x_1,x_2)\in \R^2,\ t\leq x_1 \leq t\cosh
\left(\frac{B}{t}\right),\, \vert x_2 \vert \leq B\}.
\end{equation*}
Let $t >0$ large enough to have
\begin{equation}\label{bloque}
 t\cosh \left(\frac{B}{t}\right) - t< d_K/2.
\end{equation}
Now, let us translate the domain $\Omega$ along the $x_1$ axis so that
$x_1 (\gamma)=t\cosh \left(\frac{B}{t}\right)$, recall that the side $\gamma$
is orthogonal to the $x_1$ axis. Because of inequality (\ref{bloque}) we have
that $K\subset {\rm int}\big(\pi({\mathcal C}_{t})\big)$ (see Figure \ref{picture4}).

\begin{figure}[h]
 \centerline{\includegraphics[scale =0.4]{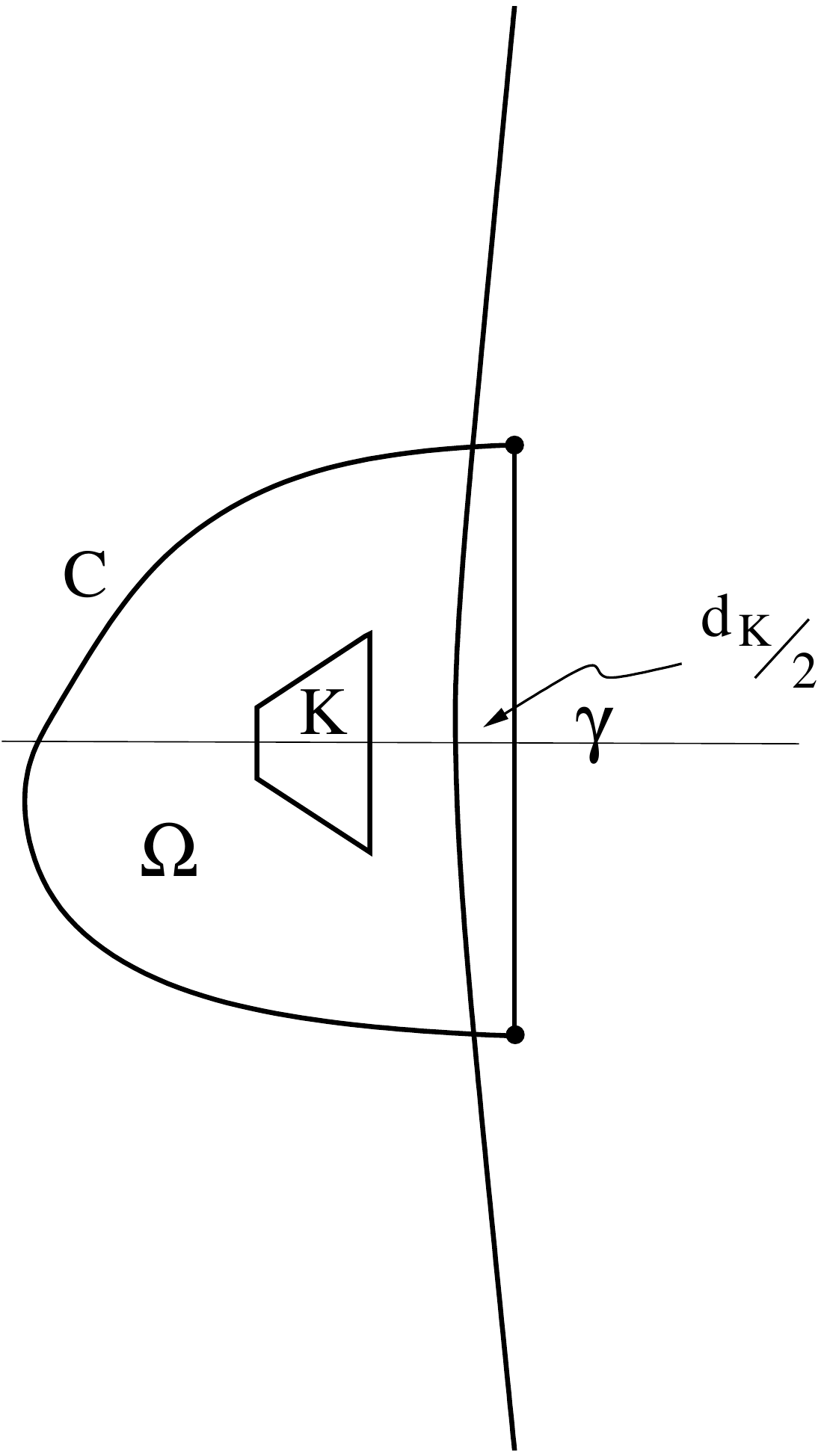}}
 \caption[]{ The subset $K$ }
 \label{picture4}
 \end{figure}

Translate vertically the catenoid $\mathcal C_{t}$ so that it stays above the
graph of $\varphi$ on $\Omega\cap  \pi({\mathcal C}_{t})$. 
Then, as the boundary of  the graph of any of the functions $u_n$ over
$\Omega\cap  \pi({\mathcal C}_{t})$ stays below the translated catenoid,
 by the maximum principle the graph of any of the  functions $u_n$ on
$\Omega\cap  \pi({\mathcal C}_{t})$ remains below this catenoid.  
This gives uniform upper  estimates
for the sequence $(u_n)$ on $K.$

\medskip

\noindent {\bf Claim 2.} {\em The sequence $(u_n)$ is strictly increasing on
int$(\Omega)$. Consequently,  $u(p_n) \to + \infty$ for any sequence $(p_n)$ of
{\rm int}$(\Omega)$ converging to an interior point of $\gamma$.}

\smallskip

The first assertion is a consequence of the {\em general maximum principle}
of H. Jenkins and
J. Serrin \cite{JS}, adapted to the metric of $Nil_3$ by
A. L. Pinheiro \cite[Theorem 2.1]{P},  because the boundary data  
of the sequence $(u_n)$  are not decreasing in $n.$    
The second assertion follows
easily. 
Uniqueness follows by  the generalization of Jenkins-Serrin result 
\cite[Section 6] {JS}  to $Nil_3$ by A. L. Pinheiro \cite[Theorem 4.2]{P}.

\medskip

Finally,  one can
use the barrier  constructed in Section \ref{barrier} in the same way as in
\cite[Theorem 3.4]{ST1}, to prove that the
function $u$ extends continuously up to
$\partial \Omega\setminus \gamma$ taking value  $u=\varphi$ on
$C\setminus \partial C$.
\qed

\begin{remark}
  Theorem \ref{Scherk-gen} also holds for functions
 $\varphi: \alpha \cup \beta \rightarrow \R$ continuous except at a finite number
 of points, where $\varphi$ has left and right limits.
\end{remark}

\subsection{Existence on Unbounded Domains}
\label{existence-unbounded-section}

 In next theorem we prove that we can solve the Dirichlet
 problem for the minimal surface equation on any convex 
 unbounded domain, different from the half-plane,
 with arbitrary piecewise continuous   boundary data,
 and on a half-plane for bounded piecewise continuous boundary data.

\begin{theo}\label{existence-unbounded}
$  $
\begin{enumerate}
\item[]{\rm{(A)}} Let $\Omega\subset\r^2$   be an unbounded convex domain
different from a half-plane.  Let $\varphi$  be a
function  on $\Gamma=\partial\Omega$ continous
 except at a  discrete set $Z \subset \partial \Omega$ of points where
$\varphi$ has left and right limits.

 Then there exists a
minimal extension of $\varphi$ over $\bar\Omega.$

\item[]{\rm{(B)}} Let $\Omega$ be a half-plane and let $\Gamma$ the
straight line that is the boundary of $\Omega.$  Let $\varphi$  be a  
bounded function  on $\Gamma$  continuous except
at a discrete set  of points where $\varphi$ has left and right limits.
 Then there exists a
1-parameter family of minimal extension of $\varphi$ over $\bar\Omega.$

\end{enumerate}

 In both cases, the boundary of the minimal extension is the union of 
 the graph of $\varphi$ with the vertical segment between
the left and the right limit of $\varphi$ at the discontinuity points.
\end{theo}

\begin{proof}
 (A) As $\Omega$ is convex and is not a half-plane,  it is either
contained   in  a wedge  or it is a strip, (in this case it has
two boundary components.)

\

\noindent {\bf Case 1.} {\em $\Omega$ is contained in a wedge.}

Assume that  the wedge  is contained in the half plane  $x_1>0.$ Let $B_n$
be the ball of radius $n,$ centered at the origin in the  $x_1$-$x_2$ plane  and let
$\Omega_n= B_n\cap \Omega.$ Denote by $\Gamma_n$ the boundary of $\Omega_n$
contained in $\partial B_n.$
Let $r_n,$ $s_n$ be the intersection points between $\Gamma_n$ and $\partial\Omega.$
Since $Z$ is discrete, we can assume that $\varphi$ is continuous at $r_n$ and $s_n.$
On  the boundary of $\Omega_n,$ we define a piecewise continuous function
$\varphi_n,$ continuous on $\Gamma_n,$
with value between $\varphi(r_n)$ and $\varphi(s_n)$
 on $\Gamma_n$ such that
\begin{equation*}
 \varphi_n(q)=\left\{\begin{array}{l}
 \varphi(q) \ \ {\rm if } \ \ q\in \partial\Omega_n\setminus\Gamma_n\\
 \varphi(r_n)  \ \ {\rm if } \ \ q=r_n \\
 \varphi(s_n)  \ \ {\rm if } \ \ q=s_n\\
 \end{array}
 \right.
\end{equation*}

As $\Omega_n$ is bounded and convex  and $\varphi_n$ is piecewise continuous,
Theorem \ref{existence} guarantees the existence of a minimal extension
$u_n$ of $\varphi_n$ on $\Omega_n.$ We recall that   the boundary of the graph
of each $u_n$
contains the vertical segment
above any  discontinuity point with endpoints the left and the right limit of
$\varphi$ at the point.
Moreover, there are no other points of the closure of the graph of $u_n$ on the
vertical geodesic
passing through the discontinuity points. 
Our  proof will be inspired in 
the case of $\h^2\times\r$ \cite{ST}.

We want to prove that there is a subsequence of $(u_n)$ converging to a
minimal solution $u$ on
every compact subset  of $\Omega$ and such that
\begin{itemize}
\item $u$ extends continuously up  to  $\ov \Omega\setminus Z$,

\item the boundary of the graph of $u$ consists in the graph of $\varphi$ over
$\partial \Omega \setminus Z$ and the vertical segments between the right limit
and the left limit of $\varphi$ at any point $p\in Z$.
\end{itemize}

By  Remark  \ref{height}, in order to prove the convergence, we only need to
prove that
uniform height estimates hold  for the sequence   $(u_n)$ on every compact
subset of $\Omega.$

Let $K$ be a compact subset  of $\Omega$ and
let  $n_0$  be such that $K\subset\Omega_{n_0}.$ Consider a horizontal
catenoid ${\mathcal C}_{\alpha(n_0)}$ such that 
$\Omega_{n_0}\subset \pi( \mathcal{C}_{\alpha(n_0)})$.
Moreover let $\wt  {\mathcal C}_{\alpha(n_0)}={\mathcal C}_{\alpha(n_0)}\cap
\{(x_1,x_2,x_3)\in\r^3  ,  \ \ |x_2|\leq B\},$
where $B$ is chosen such that $\pi( \wt  {\mathcal C}_{\alpha(n_0)})$  strictly
contains $K.$
Now, let   $N_0> n_0$ (depending only on $n_0$) such that
$\pi( \wt  {\mathcal C}_{\alpha(n_0)})\cap\Gamma_{N_0}=\emptyset.$

As, $\bar\Omega_{N_0}$ is compact,
$\varphi_{N_0}$ is bounded there, so that it is possible to translate
vertically upward  $ \wt  {\mathcal C}_{\alpha(n_0)}$ such that it is above
$\varphi_{N_0}(\partial \Omega_{N_0} \cap \partial \Omega)   $.
Notice that, for any $n\geq N_0,$    
the graph of ${u_n}_{|\partial\Omega_{N_0}}$ is below the translation  of
$ \wt  {\mathcal C}_{\alpha(n_0)}.$  This means that,  for any  $n\geq N_0,$ 
the
boundary of the graph of ${u_n}_{|\Omega_{N_0}}$stays below the
translation of $ \wt  {\mathcal C}_{\alpha(n_0)}$ and  hence, by the
maximum principle, all the graph of $u_n$ stays below  it (see figure \ref{picture5}).
This gives the uniform   height estimates from above on $K$ for every $u_n$, 
$,n\geq N_0$.

\begin{figure}[h]
 \centerline{\includegraphics[scale =0.4]{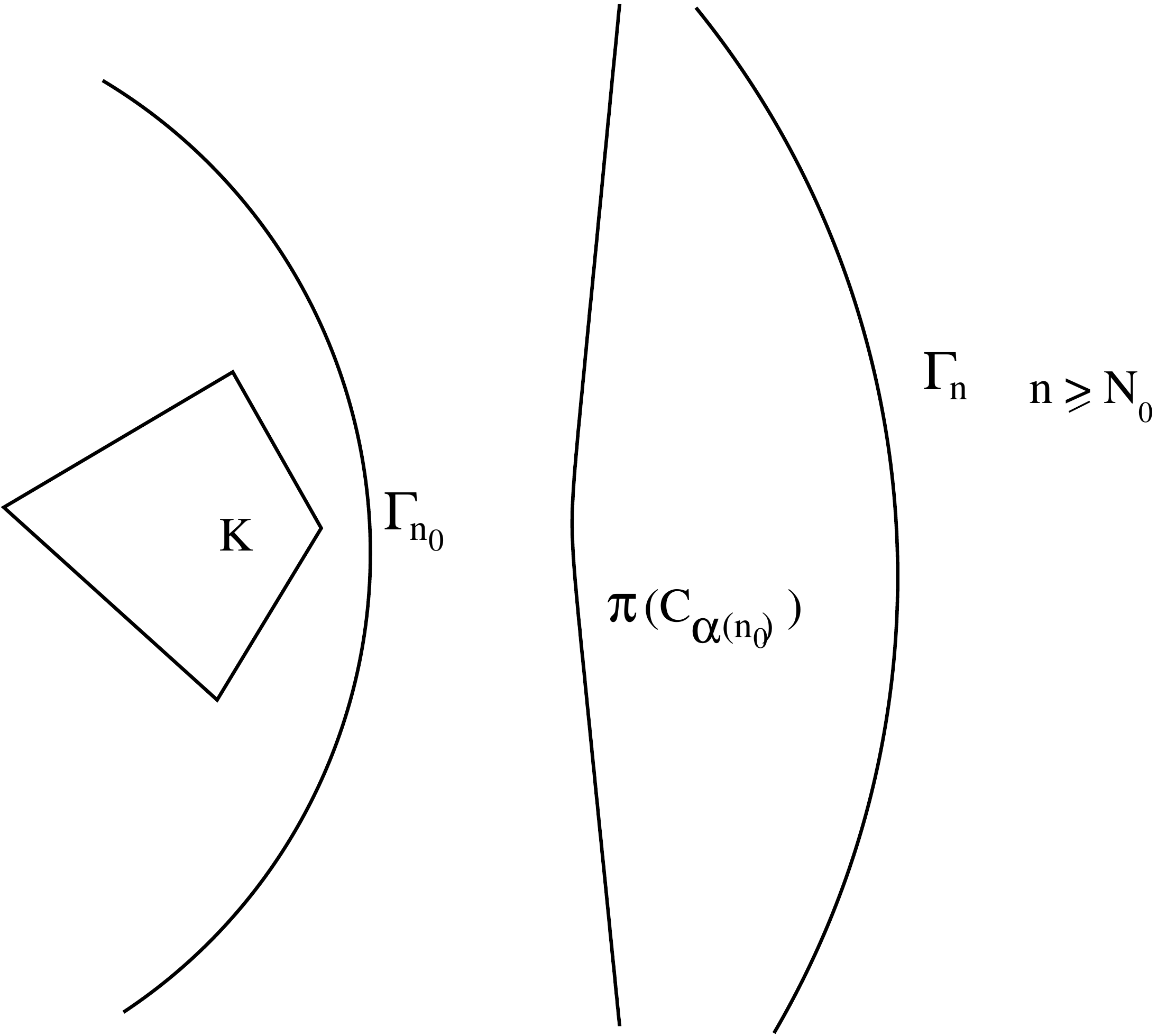}}
 \caption[]{Uniform bounds on $K:$ the graph of $\varphi_{N_0}$ over $\partial \Omega_{N_0}\cap\partial\Omega$ is below the vertical translation of  ${\mathcal C}_{\alpha(n_0)}.$}
 \label{picture5}
 \end{figure}

\smallskip

It is clear that  one finds analogously   height estimates from below 
for every $u_n$ on $K,$ $n\geq N_0.$

\smallskip

Now we prove that the function $u,$ previously defined as the limit of 
a subsequence of $(u_n),$ takes the desired boundary values.

 At any  point of $\Gamma\setminus Z$, that is  where $\varphi$ is
continuous, one can
use the barrier that we constructed before in the same way as in
\cite[Theorem 3.4]{ST1}, to prove that $u$ extends continuously up to
$\ov \Omega \setminus Z$, by  setting $u=\varphi$ on $\Gamma\setminus Z$.

 \smallskip

 Now,  we show what happens at a discontinuity point. As we noticed before, 
 the proof is analogous to the one in
 \cite[Corollary 4.1, page 325]{ST}.
  Let $p$ be a discontinuity point of $\varphi.$
Let $\varphi_1(p)$ and $\varphi_2(p)$   be the left and the
 right limit of $\varphi$ at $p.$ We assume that
 $\varphi_1(p)<\varphi_2(p).$ We   first prove that the vertical 
segment between
 $(p,\varphi_1(p))$ and $(p,\varphi_2(p))$ is contained in the boundary of the
graph of   the function $u.$ Let $l$ be such that
$\varphi_1(p)< l< \varphi_2(p).$

Let  $p_n,$ $q_n,$ distinct  points of $\Gamma$  at distance $\frac{1}{n}$ from
$p$,  such that $p_n$ and $q_n$ are not in the same component of
$\Gamma\setminus \{p\}$.
One has that, for $n$ large enough:
\begin{equation*}
u(p_n)=u_n(p_n)=\varphi(p_n)<l<\varphi(q_n)=u_n(q_n)=u(q_n).
\end{equation*}

Let $\delta_n$ be   a small arc contained in the intersection of $\Omega$
with  an  Euclidean disk  centered at $p$ of radius
$\frac{2}{n},$  with endpoints $p_n$ and $q_n$.
By the previous inequality and by the continuity of $u,$ there exists a point
$r_n\in\delta_n$   such that
$u(r_n)=l.$ This means that, for $n$ great enough, any point $(r_n,l)$ 
belongs to the graph of $u.$
Now,   when   $n\longrightarrow\infty,$ one has that $(p,l)$ belongs 
to the closure of the graph of $u.$

\smallskip

Finally, we have to prove that there  are no other points of
the closure of the graph of $u$ on the  vertical geodesic above $p$.

Let   $h$ be a real number such that
$h > {\rm max}\{\varphi_1(p), \varphi_2(p)\}$.  Then, for any $\varepsilon >0$ such that
$h-\varepsilon>{\rm max}\{\varphi_1(p), \varphi_2(p)\}$,
we construct a standard upper barrier on
a triangle
$T$ with sides $\alpha,$ $\beta,$ $\gamma.$  We choose the triangle such that
the side $\gamma$ touches $\Gamma$ at the point
$p$. We choose the  value $h-\varepsilon$ on $\gamma$ and the value $M$
on $\alpha$ and $\beta$, such that $M>\sup_{q\in
T\cap\Omega, n\in {\mathbb N}} u_n(q)$, (observe that by
construction, the functions $u_n$ are uniformly bounded on any compact subset
of $\ov \Omega$).  We denote  by $v$ the function on $T$ given by this upper
barrier.

Then, by the maximum principle, we get that $u_n(q) \leq v(q)$ for any
$q\in T\cap \ov \Omega$, $q\not= p$. Consequently we have
$u(q) \leq v(q)$ for any
$q\in T\cap \ov \Omega$, $q\not= p$. Since $v_{\vert \gamma}=h-\varepsilon$, we
obtain that the point $(p,h)$ is not in the closure of the graph of $u$.

\smallskip

We can show in the same way that any point $(p,h)$ with
$h < \min \{\varphi_1(p), \varphi_2(p)\}$ is not in the closure of the graph of
$u$.

\

\noindent {\bf Case 2.} {\em $\Omega$ is  a strip.}

 The proof is analogous  to the proof of the Case 1.
  We only  describe what are
$\Omega_n$ and $\varphi_n$ in this case.

We assume the strip is $\Omega= \{(x_1,x_2)\in\r^2 ,  \ 0<x_2<d\}$. For each
$n\in \n^*$, let $\Omega_n$ be the following rectangle:
\begin{equation}
\Omega_n=\{(x_1,x_2)\in\r^2,  \ |x_1|<n,\  0<x_2<d\}.
\end{equation}

 We choose    $\varphi_n:\partial\Omega_n\longrightarrow\R$ 
to be a 
piecewise continuous function such that
$\varphi_n(q)= \varphi(q)$
if  $q\in \partial\Omega\cap \partial \Omega_n$ and  such that it 
is monotone on  each
vertical side of $\partial\Omega_n.$

\

(B)  We can assume that the half-plane is 
$\Omega=\{(x_1,x_2)\in\r^2,  \ x_2>d\},$ $d>0.$ Let us
consider a plane $x_3=ax_2+b,$ where $a>0,$ $b>\sup_{x_2=d}\varphi.$
 For any
$n\in\n,$ we consider the
strip $\Omega_n=\{(x_1,x_2)\in\r^2,  \  d<x_2<n\}$ and the 
piecewise 
continuous function
$\varphi_n:\partial\Omega_n\longrightarrow \r$ defined as follows
\begin{equation*}
 \varphi_n(p) =
\begin{cases}
    \varphi(p) &\text{if} \ p=(x_1,d)\\
    an+b  &\text{if} \  p=(x_1,n)
\end{cases}
\end{equation*}

We solve the Dirichlet problem on $\Omega_n$  with boundary values equal to
$\varphi_n$ as in (A) and
denote by  $u_n$ the  solution.
Recall that each $u_n$ is obtained as  uniform limit on compact subsets of  the
strip, of
functions $(u_{n,k})_{k\in\n}$ defined on  rectangles
${\mathcal R}_{n,k}$ exhausting the strip $\Omega_n.$ Moreover, by the maximum
principle,
comparing the graphs of $u_{n,k}$ with vertical translation of  the plane
$x_3=ax_2+b,$  one has that 
$u_{n,k}(x_1,x_2)\leq ax_2+b,$ for all 
$(x_1,x_2)\in {\mathcal R}_{n,k},$
for all $n,$ $k.$ Therefore the graph of $u_n$ is below the plane $x_3=ax_2+b$
for every $n.$
This gives the desired estimate from above for the sequence $(u_n).$
We do the same with planes  staying below the boundary values and we find
uniform estimates from below, as well.

 As the solution that we find is contained between two planes, the existence
 of one parameter family of solutions
 is easily achieved by changing the slope of  the initial plane that
 one uses as supersolution.
\end{proof}

\begin{remark}
\label{remark-existence}
\

\begin{enumerate}
\item[]{{\rm (A)}} The existence of solutions on a half-plane can be proved in a more
general case.
Assume that the half-plane is $x_2>0$ and let
 the boundary value $\varphi$  be piecewise continuous  and
such that $\varphi(x_1,0)=cx_1$ for $|x_1|>n.$ The proof is analogous to 
the proof of Theorem \ref{existence-unbounded}(B),
using suitable tilted planes that do not
contain the $x_1$-direction.

\item[]{{\rm (B)}} In the proof of 
Theorem \ref{existence-unbounded}(A),
Case 1, we can use, as  barriers, the surfaces constructed 
in Theorem \ref{Scherk}.

\item[]{{\rm (C)}} The proof of Theorem \ref{existence-unbounded}(A)
  yields that, when  the boundary value  $\varphi$  is bounded above 
  (respectively  below)  by a constant $A,$ then the solution  given by
  our proof  is also bounded above (respectively below) by the same constant $A.$

  \item[]{{\rm (D)}}  The proof of Theorem \ref{existence-unbounded}(B)
  yields that, when the boundary value $\varphi$ is non negative, 
  then the solution  given by
  our proof  is non negative as well. Moreover, such solution has linear 
  growth  (see Definition
\ref{growth-def}).

On the contrary, it will be clear  from examples below that there are 
many unbounded solutions on a  wedge, with zero boundary value.
The existence of such surfaces means that, on a wedge, the boundedness 
of the boundary value does not imply the boundedness of the 
  extension.
 \smallskip
 Two questions arises about solutions on a strip.
 \begin{enumerate}
 \item Is the minimal solution with zero boundary 
 value on a strip, unique?
 \item Let $u$ be any minimal solution on a strip with boundary 
 value $\varphi$ such that $|\varphi|\leq M$ for some $M>0.$ Is $|u|\leq M?$
\end{enumerate}

Note that, by  \cite[Theorem 7]{MN}, any non trivial solution of  
the minimal surface equation, with zero boundary value on a strip has at least 
linear growth (see Definition \ref{growth-def}). In the same article Manzano and Nelli  prove that the growth of an entire minimal
graph in $Nil_3$ has order at most three (Theorem 6).
\end{enumerate}
\end{remark}

Let us recall the definition of growth of a graph.

\begin{definition}
\label{growth-def}

 Let $\Omega$ be an  unbounded subset of $\r^2$ and let
$\Omega_R$ be the intersection of $\Omega$ with the  ball 
of radius $R$ centered at the origin. Let $f:\r\longrightarrow \r_+$ be a  
continuous non decreasing function. The graph of  a continuous 
function $u:\Omega\longrightarrow \r$ has growth at least  
$f(R)$ (respectively at most $f(R)$) if
\begin{equation*}
\liminf_{R\longrightarrow \infty}\frac{\sup_{\partial \Omega_R}|u|}{f(R)}>0, \ \
{\rm (resp. \  \limsup_{R\longrightarrow \infty}\frac{\sup_{\partial \Omega_R}|u|}{f(R)}<\infty)}
\end{equation*}

In particular we say that  the growth is $\alpha>0$  if there 
exists a constant $c>0$ such that

\begin{equation*}
\lim_{R\longrightarrow \infty}\frac{\sup_{\partial \Omega_R}|u|}{R^{\alpha}}=c.
\end{equation*}

(linear  if $\alpha=1,$ and quadratic if $\alpha=2$)
\end{definition}

In the following we describe some examples of minimal graphs on 
unbounded domain with linear and quadratic growth.

\begin{enumerate}\label{R.wedge}
\item \label{FMP}  In  Section \ref{examples} we described the 
example  by \cite{FMP}, C. Figueroa, F.  Mercuri e R. Pedrosa . 
Let us consider 
$u_0:\R^2\longrightarrow\R$  defined by 
$u_0(x_1,x_2)=\tau x_1 x_2.$ The graph of $u_0$ is a complete minimal 
surface in 
$Nil_3(\tau)$ that is invariant by the isometry
$\varphi(x_1,x_2,x_3)=(x_1+c,x_2,x_3+c\tau x_3),$ $c\in\R.$ 
The function $u_0$ over the wedge 
$\{(x_1,x_2)\in\R^2 , \ x_1\geq 0, \ x_2\geq 0\}$ has 
zero boundary value and 
quadratic growth.

\item  In  \cite[Corollary 3.8]{C}, S. Cartier   proved that
there  exist non-zero minimal graphs on any wedge with 
angle less than $\pi,$ with zero
boundary value.  The growth of such graphs is
linear.
\end{enumerate}

In the following  theorem we show  a new example of minimal graph 
on a convex unbounded domain, containing a wedge of angle 
$\theta=\frac{\pi}{2},$ with non negative boundary value and at least  
quadratic growth.

\begin{theo}\label{P.wedge-gen}
Let  $\Omega$ be a convex, unbounded domain, different from the 
half-plane, containing a  wedge $W_{\theta}$   with vertex at the origin and 
angle
 $\theta\in \, [\frac{\pi}{2}, \pi[\,$,  and let $\varphi$ be  a 
 prescribed, non-negative, continuous data on $\partial \Omega.$ Then, there 
exists a minimal
 extension of $\varphi$ to $\Omega$  in $Nil_3(\tau),$ $\tau\not=0,$  
 with at least quadratic growth.
\end{theo}

\begin{proof} Up to an isometry, we can assume that $x_2 \geq 0$ on 
the wedge $W_{\theta}$ contained in $\Omega$ and that $W_{\theta}$ 
is symmetric with
respect to the $x_2$ axis.   Consider the graph $\Sigma$ of the function 
$v(x_1,x_2)=\tau x_1 x_2$ over the wedge 
$\{(x_1,x_2)\in\R^2 , \ x_1\geq 0, \ x_2\geq 0\}.$
We rotate $\Sigma$ in order to obtain a graph $\tilde\Sigma$ over the 
wedge $\{(x_1,x_2)\in\R^2 , \  x_2\geq| x_1|\}$ with zero boundary value. 
Moreover $x_3(p)>0$ for any
$p\in \tilde\Sigma.$

For any $n\in \n^*$ we denote by $A_n$ and $B_n$ the 
intersection points of the
boundary of $\Omega$ with the horizontal line $\{x_2 = n\}$.

Let $\Omega_n=\Omega \cap\{(x_1,x_2),   \  x_2\leq n\}. $
The set $\Omega$ is convex and  bounded.  Let  
$C_n=\partial \Omega\cap\partial \Omega_n$  and 
$\gamma_n= \partial \Omega_n\cap \{x_2=n\}.$ Observe that the 
sequence $ (\Omega_n)$ gives an exhaustion 
of the
domain $\Omega.$

Let
$\varphi_n :\partial \Omega_n \rightarrow \R$ be
a continuous function such that:
\begin{itemize}
\item $\varphi_n$ is positive and stays above $\tilde\Sigma$ along the open segment
$\gamma_n.$

\item  $\varphi_{n_{|C}}=\varphi.$
\end{itemize}
Let $u_n$ be the minimal extension of $\varphi_n$ to $\Omega_n$ given
by Theorem \ref{existence}.
Let $h_n$ be the function
whose  graph is the Scherk-type surface with boundary value zero on
$C_n$ and with value $+\infty$ on
$\gamma_n,$ given  by Theorem \ref{Scherk-gen}.

\smallskip

Finally let $K\subset {\rm int}(\Omega)$  be a compact subset and let
$n(K)\in \n^*$ such that   $K\subset \Omega_{n(K) -1}$.
 Let $L_{n(K)-1}={\rm max}_{C_{n(K)-1}}\varphi.$
By the maximum principle, for
any $n \geq n(K)$ we have $u_n \leq h_{n(K)}+L_{n(K)-1}$ on $K$.  This gives uniform
estimates  from above for the sequence $(u_n)$ on $K$. Moreover the sequence $(u_n)$ is
bounded below by $0$.

\smallskip

We deduce from Remark \ref{height} that a
subsequence
of $(u_n)$ converges to a non-negative function $u$ on int$(\Omega)$, 
satisfying  equation
(\ref{minimal}). Furthermore, the function $u$ extends continuously to 
$\partial\Omega$ with prescribed   value $\varphi.$ This can be seeing by 
comparing with the barriers described in
\ref{barrier}.

\smallskip

Observe that, for any $n,$ by the maximum principle, the graph of $u_n$
stays above $\tilde\Sigma$ , so the
graph of $u$ also stays above $\tilde\Sigma$. Since the function whose  graph is
$\tilde\Sigma$ has quadratic growth, we have that $u$ has at least quadratic
growth.
\end{proof}

\begin{remark}
Using the reflection principle and
the existence of non-zero minimal graphs over any wedge with angle
between $\pi/2$ and $\pi$ (see \cite{C} and also Proposition \ref{P.wedge}),
we get also the existence of non-zero minimal graphs over any
wedge with angle strictly between $\pi$ and
$2\pi$.
\end{remark}

 As a consequence of Theorem \ref{P.wedge-gen}, we have the following 
Corollaries.

\begin{coro}\label{P.wedge}
Let  $W_{\theta}$  be a wedge with vertex at the origin and angle
 $\theta\in \, [\frac{\pi}{2}, \pi[\,$,  and let $\varphi$ be  a prescribed, 
non-negative, 
continuous data on $\partial W_{\theta}.$ Then, there exists a minimal
 extension of $\varphi$ to $W_{\theta}$  in $Nil_3(\tau),$ $\tau\not=0,$  
 with at least quadratic growth.

\end{coro}

 The proof  of Corollary  \ref{P.wedge} is an immediate consequence of Theorem 
\ref{P.wedge-gen}, 
taking $\Omega=W_{\theta}.$
 Notice that, when the boundary value is zero,   the fact that  the  growth is 
at least 
quadratic yields that the examples of  Corollary  \ref{P.wedge} are different 
from those in \cite{C}.

\begin{coro}\label{coro-halfplane}
Let $\Pi$ be the  half-plane $\{(x_1,x_2)\in\r^2, \     x_2\geq 0\},$ and 
let $\varphi:\partial \Pi\longrightarrow \r$ an odd function.
\begin{enumerate}
\item[]{{\rm (A)}}   Assume  that  $\varphi$ is continuous and  
$\varphi(x_1)\geq 0$ 
for $x_1\geq 0.$ Then, there exists a minimal
 extension of $\varphi$ to $\Pi$  in $Nil_3(\tau),$ $\tau\not=0,$   
 with at least quadratic growth.
 \item[]{{\rm (B)}}  Assume that $\varphi$ is continuous except as a 
 discrete set of points of $\partial\Pi,$ $\varphi(x_1)\geq 0$ for $x_1\geq 0$ 
and  there exists a constant $M$ such 
that $|\varphi|\leq M.$ Then, there exists a minimal
 extension of $\varphi$ to $\Pi$  in $Nil_3(\tau)$ 
$(\tau \geq 0$, including the Euclidean 3-space$)$, that is 
bounded in $\Pi$ by the same constant $M.$ The boundary of the minimal extension 
is the union of the graph of $\varphi$ with 
the vertical segment between
the left and the right limit of $\varphi$ at the discontinuity points.

\end{enumerate}
\end{coro}

\begin{proof}   Let $Q$ be the subset of $\Pi$ given by  
$\{ (x_1,x_2)\in\r^2, \    x_1, x_2\geq 0\}.$ We  define  the function
$\psi:\partial Q\longrightarrow\r$ by $\psi=\varphi$ on 
$\partial Q\cap \{x_2=0\},$ $\psi\equiv 0$ on $\partial Q\cap   \{x_1=0\}.$

In the case (A), let $v$ be the minimal extension of $\psi$ over $Q$ 
given by Corollary \ref{P.wedge}.
 Then, extend $v$ to $\Pi$  by Reflection Principle (as in 
 \cite[Lemma 3.6]{ST4}) along 
 $\partial Q\cap   \{x_1=0\}.$ Due to the fact that $\varphi$ is odd,
 the extended function gives the desired solution with quadratic growth.

In the case (B), let  $v$ be the  minimal extension of $\psi$ over $Q$
given by  Theorem \ref{existence-unbounded}(A). Notice that, by Remark
\ref{remark-existence}, the function $v$ is bounded by the constant $M.$  
Then, extended $v$ to $\Pi$  by Reflection Principle along
$\partial Q\cap   \{x_1=0\}.$ The extend function gives the desired solution.
\end{proof}

\

\textsc{Barbara Nelli}\newline
{\em Dipartimento Ingegneria e Scienze dell'Informazione e Matematica \newline
Universit\'a di L'Aquila \newline
via Vetoio - Loc. Coppito  - 67010 (L'Aquila) \newline
 Italy}

nelli@univaq.it

\

 \textsc{Ricardo Sa Earp}
 \newline
 {\em Departamento de Matem\'atica \newline
  Pontif\'\i cia Universidade Cat\'olica do Rio de Janeiro\newline
Rio de Janeiro - 22451-900 RJ \newline
 Brazil }

rsaearp@gmail.com

\

\textsc{Eric Toubiana}\newline
{\em Institut de Math\'ematiques de Jussieu-PRG \newline 
Universit\'e Paris Diderot  \newline
UFR de Math\'ematiques  - B\^atiment Sophie Germain \newline
5, rue Thomas Mann - 75205 Paris Cedex 13 \newline
France}

eric.toubiana@imj-prg.fr


\begin{thebibliography}{99}

\bibitem{ADR} \textsc{L. Alias,  M. Dajczer,  H. Rosenberg:}
{\em The Dirichlet problem for constant mean curvature
surfaces in Heisenberg space,}
Calc. Var. Partial Differential Equations 30 (2007), no. 4, 513-522.


\bibitem{BC}\textsc{P. Berard, M.P. Cavalcante:} {\em Stability properties of rotational catenoids in the Heisenberg group,}
to appear in Mat. Contemp.  arXiv:1010.0774v3 [math.DG].



%\bibitem{BGS}\textsc{L. Barbosa, J. Gomes, A. Silveira:} {\em Foliation of 3-dimensional space forms
%by surfaces with constant mean curvature,} Bol. Soc. Bras. Mat. 18 (1987) 1-12.




%\bibitem{CH}\textsc{R. Courant, D. Hilbert:} {\em Methods of mathematical physics,}  vol.II. New
%York: Interscience (1962).

\bibitem{C} \textsc{S. Cartier:}
{\em Saddle towers in Heisenberg space,}
arXiv:1406.6610v1 [math.DG]  (2014)






\bibitem{Co} \textsc{P. Collin:} {\em Deux exemples de graphes de courbure 
moyenne constante sur une bande de $\R^2$}, \\ 
C. R. Acad. Sci. Paris Ser. I Math {\bf 311 } (1990), n. 9, 539-542.


\bibitem{CK} \textsc{P. Collin, R. Krust:}
{\em Le probl\`eme de Dirichlet pour l'\'equation des surfaces minimales sur
des domaines non born\'es,}
Bull. Soc. Math. France, 119 (1991) 443-462.


\bibitem{D}\textsc{B. Daniel:} {\em The Gauss Map of Minimal Surfaces in the
Heisenberg Group,}
International Mathematics Research Notices, Vol. 2011, No. 3, pp. 674Ð695.


\bibitem{DH} \textsc{B. Daniel,  L. Hauswirth:}  {\em Half-space theorem, embedded minimal annuli and minimal
graphs in the Heisenberg group,} Proceedings of the London Mathematical Society 98, no. 2
(2009) 445Ð470.

%\bibitem{ER}\textsc{M.F. Elbert, H. Rosenberg:} {\em Minimal graphs in  $M\times\r,$}
%Ann. Glob. Anal. Geom. 34,  (2008) 39-53.



%\bibitem{DHL}\textsc{M. Dajczer, P. Hinojosa, J.H. de Lira:} {\em Killing graphs with prescribed mean curvature,} Calc. Var. (2008) 33, 231-248.

%\bibitem{DL}\textsc{M. Dajczer, J.H. de Lira:} {\em Killing graphs with prescribed mean curvature and Riemannian
%submersions,} Ann. I. H. Poinc.  26 (2009) 763-775.

%\bibitem{DMR} \textsc{B. Daniel,  W. Meeks, H. Rosenberg:}  {\em Half-space theorems for minimal surfaces in $Nil_3$ and $Sol_3$,} (2009) 445-
%J. Differential Geometry 88 (2011), no. 1, 41-60.

%\bibitem{ENS}\textsc{M.F. Elbert, B. Nelli, R. Sa Earp:} {\em Existence of vertical ends of mena curvature $ \frac{1}{2}$ in $\h^2\times\r,$}
%TAMS, Vol. 364, n. 3 (2012)  1179-1191.

%\bibitem{FM}\textsc{I. Fernandez, P. Mira:} {\em Holomorphic quadratic differentials and the Bernstein problem in Heisenberg space,}
%Trans. Amer. Math. Soc. 361 (2009), 5737-5752.

\bibitem{Fi}\textsc{C.B. Figueroa:} {\em Geometria das subvariedades do grupo de Heisenberg,} Campinas (27/02/1996).




\bibitem{FMP}\textsc{C.B. Figueroa, F. Mercuri, R. H. L. Pedrosa:} {\em Invariant surfaces of the Heisenberg group,}
Annali di Matematica pura ed applicata
(IV), Vol. CLXXVII (1999) 173-194.

\bibitem{F} \textsc{R. Finn:} {\em Remarks relevant to minimal surfaces and to surfaces of constant mean curvature,} J. d' Analyse Math. {\bf 14} (1965), 139-160.



\bibitem{FS}\textsc{F. Fontenele, S. Silva:}  {\em A tangency principle and applications,}  Illinois J. Math. 54, 213-228 (2001).

%\bibitem{GT} \textsc{D. Gilbarg, N.S. Trudinger:} {\em  Elliptic Partial Differential Equations of Second Order,}
 %Springer, Heidelberg (1983)

%\bibitem{H}\textsc{P. Hartman:} {\em On the bounded slope condition,} Pac. J. Math. 18 (1966)  494-511.


\bibitem{JS}\textsc{H. Jenkins, J. Serrin:} {\em Variational Problems of Minimal Surface Type II.
Boundary Value Problems for the Minimal Surface Equation,} Arch. Rational Mech.
Anal. 21 (1966), 321-342.

\bibitem{JS1}\textsc{H. Jenkins, J. Serrin:} {\em The Dirichlet problem for the minimal surface
equation in higher dimensions,} Jour. fur  die reine und angewandte
Mathematik, 229 (1968) 170-187.



%\bibitem{Ki}\textsc{Y.W.  Kim, S.E.  Koh, H. Y. Lee, H. Shin, S.D.  Yang:}
%{\em Ruled minimal surfaces in the three dimensional Heisenberg group,}
%arXiv:0906.1246 [math.DG]


%\bibitem{K}\textsc{N. Korevaar:}{\em An easy proof of the interior gradient bound for solutions to  the prescribed mean curvature equation,}
%Proc. of Symphosia in Pure Math. 45 (1986) 81-89.

%\bibitem{LR}\textsc{C. Leandro, H. Rosenberg} {\em Removable singularities for sections of Riemannian
%submersions of prescribed mean curvature,}
%Bull. Sc. Math. 133, 4 (2009)  445-452.

%\bibitem{M}\textsc{A. M. Menezes:} {\em Periodic minimal surfaces in semidirect products,} ArXiv: arXiv:1309.2552v1 [math.DG] 10 Sep 2013.



\bibitem{MN}\textsc{J.M. Manzano, B. Nelli:} {\em Height and area estimates for constant mean curvature graphs in ${\mathbb E}(\kappa,\tau)$-spaces,} 	
arXiv:1504.05239 [math.DG]


\bibitem{M1} \textsc{C. B. Morrey:} {\em Multiple integrals in the calculus of variation,} Reprint of the 1966 edition, Springer (2008).



\bibitem{M2} \textsc{C. B. Morrey:} {\em The problem of Plateau on a Riemannian manifold,}  Ann. Math. No. 2(49)  (1948) 807-851.


\bibitem{NR}\textsc{ B. Nelli, H. Rosenberg}: {\sl Minimal Surfaces in
$\h^2\times\r,$}  Bull. Bras. Math. Soc. New Series 33, 2 (2002) 263-292.

 \bibitem{NR1}\textsc{ B. Nelli, H. Rosenberg}: {\sl  Errata Minimal Surfaces
in $\h^2\times\r,$ [Bull. Braz. Math. Soc., New Series 33 (2002),
263-292],} Bull. Braz. Math. Soc., New Series 38, 4 (2007) 661-664.

\bibitem{P}\textsc{A. L. Pinheiro:} {\em Minimal vertical graphs in Heisenberg
space}, Preprint (2014).


\bibitem {R-SaE} \textsc{H. Rosenberg,  R. Sa Earp:} {\em The Dirichlet problem for the minimal surface equation on unbounded planar domains,} J. Math. Pures et Appl. {\bf 68}, 1989, 163-183.

\bibitem{RST}\textsc{H. Rosenberg, R. Souam, E. Toubiana:}
{\em General Curvature Estimates for Stable $H$-surfaces in 3-manifolds and
Applications,} J.Diff. Geom. 84 (2010) 623-648.

%\bibitem{RT}\textsc{J. Ripoll, F.  Tomi:} {\em On solutions to the exterior Dirichlet problem for
%the minimal surface equation with catenoidal ends,} Advances in Calculus of Variations (2012).

\bibitem{ST1} \textsc{R. Sa Earp, E. Toubiana:}  {\em Existence and
uniqueness of minimal graphs in hyperbolic
 space}, Asian J. Math. \textbf{4} (2000) 669-694.

\bibitem{ST}\textsc{R. Sa Earp, E. Toubiana:} {\em An asymptotic theorem for minimal surfaces
and existence results for minimal graphs in $\h^2\times\r,$} Math. Ann.
\textbf{342} (2) (2008), 309--331.

\bibitem{ST4}  \textsc{R. Sa Earp and E. Toubiana},
{\em Minimal graphs in
$\hi n \times \r$ and $\r^{n+1}$},  Annales de l'Institut Fourier, 
\textbf{60} (7) (2010), 2373--2402.


%\bibitem{ST1}\textsc{R. Sa Earp, E. Toubiana:} {\em Existence and uniqueness of minimal graphs in hyperbolic space,}
%Asian J. Math. 4 (2000), no. 3, 669-693.

\bibitem{Y} \textsc{R. Younes:} {\em Minimal Surfaces in $\widetilde{PSL_2(\r)},$} Illinois Jour. of Math. 54, 2 (2010) 671-712.

\end{thebibliography}
\end{document}